\documentclass[reqno]{amsart}

\usepackage{amsmath, amssymb, latexsym, graphicx}
\numberwithin{equation}{section}

\newtheorem{theorem}{Theorem}[section]
\newtheorem{lemma}[theorem]{Lemma}
\newtheorem{proposition}[theorem]{Proposition}
\newtheorem{remark}[theorem]{Remark}
\newtheorem{corollary}[theorem]{Corollary}
\newtheorem{definition}[theorem]{Definition}
\newtheorem{conjecture}[theorem]{Conjecture}

\usepackage[usenames]{color}

\def\R{\mathcal{R}}

\newcommand{\C}{{\mathbb C}}
\newcommand{\CP}{{\mathbb CP}}
\newcommand{\Z}{{\mathbb Z}}

\usepackage{enumerate}

\begin{document}
\title[]{Algebraically integrable bodies and related properties of the Radon transform}
\author[]{Mark Agranovsky}

\address[Mark Agranovsky]{Department of Mathematics, Bar-Ilan University, Ramat Gan, 5290002, Israel}

\email{agranovs@math.biu.ac.il}

\author[]{Jan Boman}

\address[Jan Boman]{Department of Mathematics, Stockholm University, 10691 Stockholm, Sweden}

\email{jabo@math.su.se}

\author[]{Alexander Koldobsky}

\address[Alexander Koldobsky]{Department of Mathematics,
University of Missouri-Columbia,
Columbia, MO 65211, USA}

\email{koldobskiya@missouri.edu}

\author[]{Victor Vassiliev}

\address[Victor Vassiliev]{Steklov Mathematical Institute, 119991, Moscow, Russia}

\email{vavassiliev@gmail.com}

\author[]{Vladyslav Yaskin}

\address[Vladyslav Yaskin]{Department of Mathematical and Statistical Sciences, University of Alberta, Edmonton, AB T6G2G1, Canada}

\email{yaskin@ualberta.ca}

\thanks{The {third} named author was supported in part by the U.S. National Science Foundation Grant DMS-2054068. The fifth author was supported in part by NSERC. This material is partially based on the work supported by the U.S. National Science Foundation grant DMS-1929284 while the third and fifth authors were in residence at the Institute for Computational and Experimental Research in Mathematics in Providence, RI, during the Harmonic Analysis and Convexity semester program.}
\keywords{}
\subjclass[2010]{}

%%%%%%%%%%%%%%%%%%%%%%%%%%%%%%%%%%%%%%%%%%%%%%%%%%%%%%%%%%%%%%%%%%%%
\begin{abstract}
Generalizing Lemma 28 from Newton's ``Principia" \cite{Newton}, Arnold \cite{Ar} asked for a complete characterization of algebraically integrable domains.
In this paper we describe the current state of Arnold's problems. We also consider closely related problems about the Radon transform of indicator functions.

\end{abstract}

\maketitle

\section{Introduction}

The questions considered in this survey belong to the area of geometric tomography (see the book \cite{RG}) which lies at the crossroads between convex geometry and integral geometry and can be defined as the study of geometric properties of solids based on data about their sections and projections.

We study algebraic properties of two important {volumetric characteristics} in geometric tomography.
For a body (compact set with non-empty interior) $K$ in $\mathbb R^n,$ $\xi\in S^{n-1}$ and $t\in \mathbb R,$
the {\bf cutoff functions} of $K$ represent the $n$-dimensional volume of the parts of $K$ cut by the
hyperplane perpendicular to $\xi$ at distance $t$ from the origin:

\begin{equation}\label{E:VK}
\begin{aligned}
V^{+}_K(\xi,t)& ={\mathrm {Vol}}_{n}(K\cap \{x\in \mathbb R^n: \ \langle x,\xi\rangle \le t\}) = \int_{K\cap \{x\in \mathbb R^n:\ \langle x, \xi \rangle  \le t\}} dx, \\
V^{-}_K(\xi,t)& ={\mathrm {Vol}}_{n}(K\cap \{x\in \mathbb R^n: \ \langle x,\xi\rangle \ge t\}) = \int_{K\cap \{x\in \mathbb R^n:\ \langle x, \xi \rangle  \ge t\}} dx,
\end{aligned}
\end{equation}
The {\bf section function} of $K$ is the $(n-1)$-dimensional volume of the section of $K$ by the same hyperplane:
\begin{equation}\label{E:AK}
\begin{aligned}
A_K(\xi,t)={\mathrm {Vol}}_{n-1}(K\cap \{x\in \mathbb R^n:\ \langle x,\xi\rangle=t\})&=\R(\chi_K)(\xi,t) \\
&=\int_{K\cap \{x\in \mathbb R^n:\  \langle x, \xi\rangle  =t\} } dx.
\end{aligned}
\end{equation}
Here $\R$ stands for the Radon transform, $\chi_K$ is the indicator (characteristic function) of $K,$ $\langle x, \xi \rangle$ is the inner (scalar) product in $\mathbb R^n$, and $dx$ is Lebesgue measure on $\mathbb R^n$ or $\{x:\ \langle x,\xi \rangle =t\},$ correspondingly. Clearly, the cutoff functions and the section function are related via differentiation in $t$.

Most of our problems take root in Lemma {28} about ovals from Newton's Principia \cite{Newton}; see also the discussion in \cite{ArVas, Wi}.
Newton proved that if $K$ is a convex infinitely smooth domain in $\mathbb R^2,$ then the cutoff function of $K$ cannot appear as the solution of a polynomial equation involving the parameters of the cutting hyperplane.
Formalizing  the question and extending it to higher dimensions,
Arnold \cite{Ar} {asked whether there exist domains with smooth boundaries in $\mathbb R^n$ (apart from ellipsoids for odd $n$) for which the  cutoff functions $V_K^{\pm}$ are branches of an algebraic function.} Recall that a function $f(\xi,t)$ is {\bf algebraic} if there exists a {non-zero} polynomial $\Phi(\xi,t,w)$ of $n+2$ variables such that $$\Phi(\xi,t, f(\xi,t)) \equiv 0.$$

\begin{definition} \label{D:alg-int} (cf. \cite{Ar,ArVas,Vas}) {A domain $K$ is {\bf algebraically integrable} if the two-valued cutoff function $V^{\pm}_K(\xi,t)$ coincides with some branches of an algebraic function.}
\end{definition}
In Section \ref{S:S1}, we present the current state of Arnold's problems. In particular, it was proved in \cite{Vas} that there are no {algebraically integrable bodies with infinitely smooth boundaries} in even dimensions. However, the odd-dimensional case is still open.
\smallbreak

{In Sections 3-7, we consider similar questions that are motivated by Arnold's problem {and address the single-valued section function $A_K(\xi,t)$ rather than the multi-valued cutoff function $V_K(\xi,t).$  Therefore, we study geometric properties of bodies $K$ from the point of view of algebraic properties of their Radon transform $A_K$}. The following definition is similar to Definition \ref{D:alg-int}:}

\begin{definition} \label{D:algRT} Let $K$ be a body in $\mathbb R^n.$  We say that $K$ has {\bf algebraic Radon transform} if
there exists a function $\Psi \in C(S^{n-1})[t,w]$ which is an element of the polynomial ring of two variables over the algebra $C(S^{n-1})$ (i.e.
it is a polynomial with respect to $t,w$ with coefficients which are continuous functions of $\xi)$ and satisfies the equation
$$\Psi(\xi,t, A_K(\xi,t)) =0$$
for every $t$ such that the hyperplane $\langle \xi, x \rangle =t$  intersects $K.$
\end{definition}

 The essential difference between the two definitions is that in Definition \ref{D:algRT} we do not assume that $\Psi$ is a polynomial in $\xi,$ as we  do in Definition \ref{D:alg-int}, so the section function $A_K(\xi,t)$ is algebraic only with respect to the variable $t.$

Note that if $K$ is algebraically integrable  (i.e., the cutoff function $V_K^{\pm}$ is algebraic),
then the section function $A_K(\xi,t)$ is also algebraic as the derivative
$$A_K(\xi,t)= {\pm} \frac{d}{dt}V_K^{\pm}(\xi,t)$$
of an algebraic function. Thus, the class of domains with algebraic Radon transform is larger than that of algebraically integrable domains.

Our basic example is the unit ball $B^n$ in $\mathbb R^n.$ In this case
$$A_{B^n}(\xi,t) =\frac{\pi^{\frac{n-1}2}}{\Gamma(\frac {n+1}2)} (1-t^2)^{\frac{n-1}{2}}.$$
If $n$ is odd then $A_{B^n}(\xi,t)$ is a polynomial in $t.$  Applying an affine transformation to $B^n$ we obtain that $A_K(\xi,t)$ is a polynomial in $t$ if
$n$ is odd and $K$ is an ellipsoid.

In this article, we consider classes of bodies $K$ satisfying Definition \ref{D:algRT} with the defining polynomial $\Psi$ of a certain form.
The property of ellipsoids in odd-dimensional spaces mentioned above gives rise to the following:

\begin{definition} [\cite{Ag}] \label{D:poly-int} Let $K$ be a domain in $\mathbb R^n.$ We call $K$ {\bf polynomially integrable} if the Radon transform $A_K(\xi,t)$ of $\chi_K$ is a polynomial with respect to  $t$ when the corresponding hyperplane intersects $K.$
\end{definition}

In the case of polynomially integrable domains, the equation $\Psi(\xi,t,w)=0$ in Definition \ref{D:algRT} has the form $\Psi(\xi,t,w)=w- \sum_{k=0}^N a_k(\xi)t^k=0.$
Another example is given by {\bf rationally integrable domains} where $A_K(\xi,t)$ is the ratio of polynomials in $t$: $A_K(\xi,t) = \frac{P(\xi,t)}{Q(\xi,t)}$ and, correspondingly,
$\Psi(\xi,t,w)= Q(\xi,t)w-P(\xi,t).$
\smallbreak

A polynomially integrable body is  algebraically integrable if we additionally demand that $A_K(\xi,t)$ extends from  the unit sphere $|\xi|=1$  to $\mathbb R^n$ as a polynomial, when $t$ is fixed. However, in Definition \ref{D:poly-int} no essential condition is imposed on the behavior of $A_K(\xi,t)$ with respect to $\xi$ hence the two classes are different, although they intersect.
\smallbreak

In Section \ref{S:S3}, we describe the result of \cite{Ag, KMY} that the only  polynomially integrable bodies are ellipsoids in odd dimensions. In Section \ref{S:S4}, we extend this result
to the case where the section function is real analytic, in particular, it can be a rational function without real poles. In Section \ref{S:S5}, a relation between polynomial integrability and finite stationary phase expansions of certain Fourier integrals is established. This relation is used to characterize  locally polynomially integrable hypersurfaces.  In Section \ref{S:S6}, domains with algebraic $X$-ray transform (length chord functions) are studied. Finally, in Section \ref{S:S7} we present Theorem \ref{jbthm} showing that the Radon transform of a compactly supported distribution can be supported in the set of tangent planes to the boundary $\partial D$ of a bounded convex domain $D \subset \mathbb R^n$ only if $\partial D$ is an ellipsoid. This result gives a new proof of the fact that polynomially integrable bodies must be ellipsoids (Theorem 3.2).

\section{Algebraically integrable bodies in Euclidean space} \label{S:S1}

\subsection{Problems and main results} \label{S:Vas}

By a theorem of Archimedes (see \cite{arch}, \cite{bo21}), spheres in $\mathbb R^3$ are algebraically integrable. Indeed, the volume cut from the unit ball
in ${\mathbb R}^{3}$ by a hyperplane at distance
$t<1$ from the origin is a polynomial in $t$. It is easy to check that the same is true for arbitrary ellipsoids in odd-dimensional spaces. On the contrary, Newton's result \cite[Lemma 28]{Newton} mentioned in the Introduction asserts that there are no convex algebraically integrable bodies with smooth boundaries  in $\mathbb R^2$.

V.~Arnold \cite[Problems 1987-14, 1988-13, 1990-27]{Ar} conjectured that there are no algebraically integrable bodies with smooth boundaries in  even-dimen\-si\-onal spaces and asked whether there exist such  bodies other than ellipsoids in odd-dimensional spaces. The even-dimensional conjecture was confirmed in \cite{Vas}.

\begin{theorem}[\cite{Vas}]
\label{even}
There are no algebraically integrable bodies with $C^\infty$-smooth boundaries in even-dimensional spaces.
\end{theorem}

The odd-dimensional case is still open; see statements \ref{mthm2}--\ref{remsec} below for some partial results towards it. \smallskip

\noindent
{\bf Remarks.}  1. By projective duality and the Tarski--Seidenberg theorem, if a body in $\mathbb R^n$ is algebraically integrable, then its boundary is semi-algebraic. Therefore, it is enough to consider the case when our body is bounded by a smooth component of a hypersurface defined by a polynomial equation $F(x) =0$.

2. The condition of infinite smoothness is essential in this problem: for an arbitrary natural $N$ there exist algebraically integrable bodies with $C^N$-smooth boundaries in even-dimensional spaces.

3. In fact, we prove even more: under the conditions of Theorem \ref{even} the analytic con\-ti\-nu\-ation of the volume function to the space of complex hyperplanes in $\C^n$ cannot be even algebroid, because it necessarily takes infinitely many different values at the same hyperplanes.
\medskip

\begin{conjecture}
\label{mconj}
For any odd number $k$, even number $m$ and $\varepsilon \in (0,1)$, the body in $\mathbb R^{k+m}$ bounded by the hypersurface
\begin{equation}
\label{mex}
\left(\sqrt{x_1^2 + \dots + x_k^2} -1 \right)^2 + y_1^2 + \dots + y_m^2 = \varepsilon^2
\end{equation}
$($i.e. the $\varepsilon$-neighborhood of the unit sphere $S^{k-1} \subset \mathbb R^k \subset \mathbb R^{k+m})$
is  algebraically integrable.
\end{conjecture}

It is encouraging that the obstruction to algebraic integrability mentioned in the third remark (the infinite ramification of the analytic continuation of volume function) fails for this body.

\begin{theorem}[\cite{ne}]
\label{mthm2}
The body introduced in Conjecture \ref{mconj} is algebroidally integrable (i.e. its cutoff functions are algebroid). In particular, the analytic continuation of this function from any domain of the space of real hyperplanes where this function is regular is finite-valued.
\end{theorem}

So, to prove Conjecture \ref{mconj}, it suffices to check that this analytic continuation has only power growth at its singular points. {Even if this conjecture was confirmed one more Arnold's question would remain unsolved, namely, Problem 1990-27 of \cite{Ar}} asking whether there are {\em convex} algebraically integrable bodies in $\mathbb R^{2k+1}$ except for ellipsoids. \medskip

In any case, algebraically integrable bodies in $\mathbb R^{2k+1}$ are very rare. In  particular, the local geometry of their boundaries satisfies very strong conditions.

\begin{theorem}[\cite{VA}, \cite{APLT}]
\label{prep}
If a body $K \subset \mathbb R^{2k+1}$ is algebraically integrable, then

1$)$ the inertia indices of the second fundamental form of its boundary are even at all points where this form is non-degenerate;

2$)$ the algebraic closure of this boundary in $\C^n$ has no tame parabolic points.
\end{theorem}

Recall that a regular point of an affine hypersurface is {\em parabolic} if the second fundamental form is degenerate at this point; a parabolic point is called {\em tame} if the tangent hyperplane at this point has no other tangencies with the hypersurface in a neighborhood of this point.

\begin{remark} \label{remsec} \rm
If either of the two conditions of Theorem \ref{prep} is not satisfied, then not only the cutoff function $($\ref{E:VK}$)$ is not algebraic, but also the section function $($\ref{E:AK}$)$ is not algebraic.
\end{remark}

For additional restrictions on the geometry of algebraically integrable bodies, see \cite[Chapter III]{APLT} and  \cite[Chapter 7]{bo21}.

\medskip
In addition, we consider the property of {\em local} algebraic integrability. The volume function is regular analytic on the set of hyperplanes transversal to the boundary of the body; the set of tangent hyperplanes divides this set into several connected components. For example, the body bounded by the surface (\ref{mex}) with arbitrary $k$ and $m$ has four such components: the hyperplanes  from them intersect the hypersurface (\ref{mex}) along manifolds diffeomorphic to a) the empty set, b) $S^{k+m-2}$, c) $S^{k-1}\times S^{m-1}$, d) $S^{k-2} \times S^{m}$.

We call such a component a {\em lacuna} if the volume function coincides in it with an algebraic function.  A trivial example of a lacuna is the domain consisting of hyperplanes not intersecting the body. We show that this example is not unique even in the even-dimensional case.

\begin{proposition}[see \cite{matnot}] \label{mprop}
If $m$ is even $($and $k$ is arbitrary$)$, then the component of the set of generic hyperplanes  containing the hyperplane $x_1=0$ is a lacuna of the body bounded by the hypersurface
\begin{equation}
(x_1^2+\dots +x_k^2 -1)^2 + \left(y_1^2 + \dots + y_m^2\right)= \varepsilon ^2. \label{tnb}
\end{equation}
\end{proposition}

In the case $k=3$ a more general class of examples is given in \cite{VA}. \medskip

\begin{remark}
\label{rempet} \rm
There is a deep analogy between this set of problems and Petrovsky's theory of lacunas of hyperbolic partial differential equations and systems (developed further by Leray, G$\mathring{\rm a}$rding, Atiyah, Bott a.o.). In particular, the radical difference in the behavior of both volume functions and solutions of hyperbolic PDE's in spaces of different parity of dimensions is explained by the fact that the intersection form in middle homology groups of complex varieties (which is the main part of {\em Picard--Lefschetz formulas} controlling the ramification of integrals) is symmetric or antisymmetric depending on the parity of the dimension. $G{\aa}rding$
\end{remark}

\subsection{Integrability and Picard--Lefschetz theory}

Let $K$ be a body in $\mathbb R^n$, the boundary of which is a smooth component of the hypersurface defined by a polynomial equation $F(x)=0$. Let
$A \subset \C^n$ be the set of complex zeros of this polynomial $F$. A complex affine hyperplane $X \subset \C^n$ is {\em generic} if its closure $\bar X \subset \CP^n$ is transversal to the stratified variety $ A \cup \CP^{n-1}_\infty$ where $\CP^{n-1}_\infty \equiv \CP^n \setminus \C^n$. Denote by $P_n$ the space of all affine hyperplanes in $\C^n$, and by $\Sigma $ its subset  consisting of non-generic hyperplanes.

By Thom's isotopy lemma (see e.g., \cite{GM}), pairs $(\C^n, A \cup X)$ form a locally trivial fiber bundle over the space $P_n \setminus \Sigma$ of generic hyperplanes $X$. In particular, there is a vector bundle over $P_n \setminus \Sigma$, whose fiber over a point $\{X\}$ is the relative homology group
\begin{equation}
\label{homo}
{\mathcal H} (X) \equiv H_n(\C^n, A \cup X; \C).
\end{equation}
Moreover the latter bundle admits a natural local trivialization (called {\em the Gauss--Manin connection}) defined by covering homotopy of relative cycles in the fibers $(\C^n, A \cup X)$ of the former bundle.
The group \begin{equation}
\label{funda}
\pi_1(P_n \setminus \Sigma, \{X\})
\end{equation} acts on the group (\ref{homo}) by monodromy operators defined by this connection. Explicit formulas for this action are provided by {\em the Picard--Lefschetz theory}; see e.g., \cite{AVG},  \cite{APLT}.

It is easy to see that  integrals of the volume form \begin{equation}
dx_1 \wedge \dots \wedge dx_n
\label{volu}
\end{equation} along the elements of the group (\ref{homo}) are well-defined and form a  linear function on this group for any $X$.

For any hyperplane $\{X\} \in P_n \setminus \Sigma$ and any element $\gamma \in {\mathcal H}(X)$, define a function on any simply-connected neighborhood of the point $\{X\}$ in $P_n$ as follows: its value at the point $\{X'\}$ is equal to the integral of the form (\ref{volu}) along the element of the group
$
{\mathcal H} (X')$ obtained from $\gamma$ by the Gauss--Manin continuation over an arbitrary path connecting $\{X\}$  and $\{X'\}$ in our neighborhood. This function is holomorphic and so can be continued to an analytic function on the entire set $P_n \setminus \Sigma$.

If a hyperplane $X$ is {\em real} (i.e., its intersection with ${\mathbb R}^n \subset \C^n$ is a hyperplane in $\mathbb R^n$), and $\gamma$ is the homology class of one of parts cut by $X$ from our body $K$, then this analytic function coincides in the set of neighboring {\em real} hyperplanes $\{X'\} \approx \{X\}$ with one of the two branches of the volume function participating in the definition of algebraic integrability. If our body is algebraically integrable, then this analytic function is algebraic and, in particular, its analytic continuation to the space $P_n$ of complex hyperplanes is finitely-valued.

Therefore, to prove the non-integrability of a body it is enough to present a real generic hyperplane $X$ such that the integrals of the volume form take infinitely many values on the orbit of this element $\gamma \in {\mathcal H}(X)$ under the monodromy action of the group (\ref{funda}).

\subsubsection{Example: convex case}

Let $n$ be even,
$F: \mathbb R^n \to \mathbb R$ be a polynomial, and $K$ be a bounded {\em convex} connected component of the subset in $\mathbb R^{n}$ where $F \leq 0$; suppose that its boundary $\partial K$ is smooth. The restriction of any linear function $L: \mathbb R^{n} \to \mathbb R$ to $\partial K$ has exactly two critical points. By Sard's lemma we can choose $L$ in such a way that these critical points will be Morse. Denote by $m$ and $M$ the minimal and maximal values of this restriction. For any generic value $t \in (m,M)$ denote by  $[K]_-(t)$ and $[K]_+(t)$ two elements of the group  ${\mathcal H}(L^{-1}(t))$ defined by the positively oriented domains
 \ $K \cap \{x| L(x) \leq t\}$ \ and \ $K \cap \{x| L(x) \geq t\}$ \ respectively. Fix a  generic point $t_0$ of the interval $(m,M)$ so that the hyperplane $L^{-1}(t_0)$ does not belong to $\Sigma$. Let $\alpha$ and $\beta$ be two elements of the group $\pi_1(P_n \setminus \Sigma, L^{-1}(t_0))$ defined by {\em pinches} related to the segments $[m,t_0]$ and $[t_0, M]$ (that is, loops consisting of hyperplanes $L^{-1}(t)$, where $t \in \C^1$ goes from $t_0$ to a very small neighborhood of the point $m$ or $M$ along the segment, then turns in the positive direction around this point and comes back to $t_0$ along the same path).

\begin{lemma}
\label{lem1}
Monodromy along the  loop $\alpha$ moves the class \ $[K]_-(t_0)$ \ to  \ $-[K]_-(t_0)$.  Monodromy along $\beta$ moves \ $[K]_+(t_0)$ to \ $-[K]_+(t_0)$.
\end{lemma}

\noindent
This lemma easily follows from the Picard--Lefschetz formula; see e.g., \cite{APLT}.
\medskip

Of course, these loops (and arbitrary elements of $\pi_1(P_n \setminus \Sigma, L^{-1}(t_0))$) do not change the cycle $[K] \equiv [K]_-(t_0)+ [K]_+(t_0)$
which defines an element of the groups ${\mathcal H}(X)$ for all $X$ simultaneously.

Denote by $v(t)$ the volume of the domain $[K]_-(t)$.

\begin{corollary}\label{cor1}
The analytic continuation along the loop $\alpha$ $($respectively, $\beta)$ moves  the function $v(t)$ to $-v(t)$ $($respectively, to $2V-v(t)$, where $V$ is the volume of the entire domain $K)$. In particular, for any integer $p$ the analytic continuation along the loop $(\alpha  \beta)^p $ moves $v(t)$ to $v(t) +2pV$.
\end{corollary}

So, these  continuations take infinitely many values at one and the same point  $L^{-1}(t_0)   \in P_{n} \setminus \Sigma$, and the function $v$ cannot be algebraic.
\medskip

An explicit construction of the loop in $P_n \setminus \Sigma$ increasing the volume function by twice the volume of the body (and hence proving  the non-algebraicity of this function) can be presented also for arbitrary bodies with smooth boundaries in $\mathbb R^2$; see \cite{Vas}. For greater even $n$ and general (non-convex) bodies, we have only a non-constructive proof of Theorem \ref{even}, based on the theory of reflection groups; see the next section.

\subsection{Outline of the proof of Theorem \ref{even}}

\subsubsection{General scheme}
\label{sche}

Let $K$ be an arbitrary domain in $\mathbb R^n$ ($n$ even) bounded by a $C^\infty$-smooth component $\partial K$ of the set $\{x| F(x)=0\}$. Again, let $L: (\C^{n},\mathbb R^n)  \to (\C,\mathbb R) $ be a real linear function, the restriction of which to $\partial K$ is strictly Morse.

Starting from these data, we will construct an integer lattice $\Z^r$, a $\Z$-valued bilinear form $\langle \cdot, \cdot \rangle$ on it, and a system of generators of $\Z^r$ (corresponding to all critical points of $L$ on $\partial K$). If $K$ is algebraically integrable, then the subgroup of the orthogonal group of the space $\Z^r \otimes \mathbb R$ generated by reflections in hyperplanes orthogonal (in the sense of our bilinear form) to these generators should be finite, i.e., to be a Weyl group. All Weyl groups are well-known; it is known, in particular, that they do not admit non-trivial elements of the lattice which are invariant under all reflections. On the other hand, we will present such an invariant element and so get an obstruction to integrability of $K$.

\subsubsection{Lattice}
Let $m=m_1 < m_2 < \dots < m_q = M$ be all critical values of the function $L|_{\partial K}$, and $t_0 \in [m,M]$ be a generic value so that the hyperplane $L^{-1}(t_0)$ does not belong to $\Sigma \subset P_n$. Let $O_1, \dots, O_q \in {\mathbb R}^{n}$ be corresponding critical points. By the Morse lemma, for a small ball $B_j \subset \C^n$ centered at any of these points $O_j,$ and a sufficiently small (compared with the size of $B_j$) positive number $\varepsilon$, all groups $H_n(B_j, B_j \cap (A \cup L^{-1}(m_j+ \tau)))$, $\tau \in (0, \varepsilon)$, are isomorphic to $\Z$ and are generated by some relative cycles $\Delta_j(\tau)$ called {\em vanishing cycles}.
Let us fix arbitrarily an orientation of these vanishing cycles which depends continuously on $\tau$ and consider the function $v_j$ on the interval $(m_j, m_j+ \varepsilon)$, whose value  at the point \ $m_j +\tau$ \ is equal to the integral of the  form  (\ref{volu}) along the cycle  $\Delta_j(\tau)$. This function is analytic there; its values on the interval are real or purely imaginary depending on the parity of the Morse index of the critical point $O_j$ of the function $L|_{\partial K}$. The rotation of $\tau$ around the origin in $\C^1$ moves the vanishing cycle $\Delta_j(\tau)$ to minus itself, therefore the function $w_j(\tau) \equiv v_j(m_j+\tau)$ splits on the interval $(0,\varepsilon)$ into a power series in half-integer (but not integer) powers of $\tau$.

Let us connect a distinguished point $m_j+\tau_j$ of each interval $(m_j, m_j+\varepsilon)$ and the non-critical value $t_0$ by a path in $\C^1$ going along the real line in the upper half-plane. Let $\bar v_j$ be the germ at the point $t_0$ of the analytic continuation of the function $v_j$ along this path: its value at $t_0$ is equal to the integral of the form (\ref{volu}) along the element $\bar \Delta_j \in {\mathcal H}( L^{-1}(t_0))$ obtained from the vanishing cycle $\Delta_j(\tau_j)$ (considered as an element of the group ${\mathcal H}(L^{-1}(m_j+\tau_j))$) by the Gauss--Manin connection along this path.

Consider the group $\Z^q$ of formal linear combinations of germs $\bar v_j$ with integer coefficients. The obvious evaluation homomorphism maps this group into the space of germs of holomorphic functions at $t_0$. The lattice $\Z^r$ promised in \ref{sche} is the image of this homomorphism.

\subsubsection{Bilinear form and reflection group}
Define first a bilinear form on the lattice $\Z^q$ of formal linear combinations of germs $\bar v_j$. Consider the chain of homomorphisms
\begin{equation}
\label{mv}
\Z^q \to H_n(\C^n, A \cup L^{-1}(t_0)) \to H_{n-1}(A \cup L^{-1}(t_0)) \to H_{n-2}(A \cap L^{-1}(t_0)),
\end{equation}
the first of which
maps any formal sum $\sum \alpha_j \bar v_j$ to the homology class of the cycle $\sum \alpha_j \bar \Delta_j$, the second  is the boundary operator, and the third is the differential of the Mayer--Vietoris exact sequence.
The bilinear form in the lattice $\Z^q$ is lifted by this composite map from the intersection form in the (smooth part of) $(n-2)$-dimensional complex variety $A \cap L^{-1}(t_0)$.

\begin{lemma}\label{lem2}
This bilinear form can be lowered to the lattice $\Z^r$.
\end{lemma}

\noindent
{\em Proof.} Suppose that a linear combination $\sum \alpha_j \bar v_j$, $\alpha_j \in \Z$, defines the zero germ at $t_0$, but its pairing
$\left \langle \sum \alpha_j \bar v_j, \bar v_l \right \rangle$
with some element $\bar v_l$ is a non-zero number $C$. Consider the ``pinch'' loop in $\C^1$ starting and ending at $t_0$, embracing the critical value \ $m_l$ \ and running twice along our path connecting the points $t_0$ and $m_l +\tau_l$. This loop defines an element of the group $\pi_1(P_n \setminus \Sigma, L^{-1}(t_0))$: any point $t \in \C^1$ is associated with the hyperplane $L^{-1}(t)$.  According to the Picard--Lefschetz formula, the analytic continuation of our {\em zero function} $\sum  \alpha_j \bar v_j$ along this loop adds to it the (definitely non-zero) function $\bar v_l$ with coefficient $\pm C \neq 0$. \hfill $\Box$ \medskip

Such  analytic continuations of the functions $\bar v_j$ along all $q$ pinch loops preserve the lattice $\Z^r$. By the Picard--Lefschetz formula, they act on this lattice as reflections in hyperplanes orthogonal to corresponding elements $\bar v_l$ with respect to our bilinear form, in particular the pinch corresponding to \ $m_l$ \ moves \ $\bar v_l$ \ to \ $-\bar v_l$. Consider the subgroup of the orthogonal group of $\Z^r \otimes {\mathbb R}$  generated by these $q$  reflections.

\subsubsection{If $K$ is integrable then this reflection group is finite}

\begin{lemma}[see \cite{Vas}]
\label{lem3}
The class of the domain $K \cap \{x| L(x) \leq t_0\}$ $($respectively, $K \cap \{x| L(x) \geq t_0\})$ in the group ${\mathcal H}( L^{-1}(t_0))$ is equal to the sum of $($appropriately oriented$)$ the vanishing cycles $\bar \Delta_j$ over all $j$ such that $m_j < t_0$ $($respectively, $m_j>t_0)$.
\end{lemma}

If $K$ is algebraically integrable, then the volume of the domain $K \cap \{x| L(x) \leq t\}$ should be an algebraic function of $t$, hence the sum $\sum \bar v_j$ over $j$ such that $m_j < t_0$ should have a finite orbit under the action of our reflection group in the space $\Z^r$. Replacing $t_0$ with a point $t'_0$ from another interval of non-critical values in the segment $[m,M]$ we prove the analogous statement for the sum of similar germs $\bar v'_j $ at the point $t'_0$ over all $j$ such that $m_j < t'_0$. Identifying then spaces of germs at points $t_0$ and $t'_0$ by analytic continuation along a path between these points in the upper half-plane of $\C^1$ we prove that all sums $\sum_{j \leq s} \bar v_j $  for arbitrary $s=1, \dots, q$ have finite orbits under our reflection group in $\Z^r$.

Therefore, also the orbits of all particular generators $\bar v_j$ of this lattice should be finite, which implies the finiteness of the entire reflection group. \medskip

\begin{remark} \rm
The group  $\pi_1(P_n \setminus \Sigma, L^{-1}(t_0))$ acts transitively on the set of
all vanishing cycles $\bar \Delta_j$ (although the action of only its subgroup generated by our pinch loops may be not sufficient for this), see \cite{Vas}.
\end{remark}

\subsubsection{Invariant element}

By Lemma \ref{lem3}, the sum of all $q$ function germs $\bar v_j$ is the constant function equal to the volume of the entire body $K$. This volume is positive, therefore this sum is a non-zero element of the lattice $\Z^r$. On the other hand, it is invariant under all our reflections: indeed, the homomorphism (\ref{mv}) is obviously trivial on it; moreover a non-trivial action on it of some reflection would imply a non-trivial ramification of the constant function. Therefore, our reflection group cannot be finite: otherwise it would be one of the (well-known) Weyl groups that do not admit non-trivial invariant lattice elements.

\subsection{On proofs of other statements}

The proof of Theorem \ref{mthm2} consists of an explicit calculation of the monodromy action of the group $\pi_1(P_n \setminus \Sigma, \{X\})$ on the space ${\mathcal H}(X)$: the common orbit of all vanishing cycles which can participate in domains cut by hyperplanes from the body (\ref{mex}) consists of exactly four elements. Theorem \ref{prep} and Remark \ref{remsec} follow from the {\em local} monodromy theory of isolated function singularities: the violation of either of its two conditions at a point of $A$  implies a logarithmic ramification of the analytic continuation of the volume function in an arbitrary neighborhood of the tangent hyperplane at such a point.
Proposition \ref{mprop} is proved in \cite{matnot} by explicit calculation of integrals.

\section{Polynomially integrable convex bodies} \label{S:S3}

In this section we completely characterize infinitely smooth (having infinitely smooth boundary) polynomially integrable bodies; see Definition \ref{D:poly-int}. Theorem \ref{odd-dim}
immediately implies that the only such bodies in odd dimensions are ellipsoids, as it was proved in \cite{KMY}. On the other hand,
Theorem \ref{even-dim} generalizes the result from \cite{Ag} that there are no such bodies in even dimensions.
\smallbreak

Let $K$ be an  infinitely smooth  convex body in $\mathbb R^n$ that is polynomially integrable, i.e.,
$$	
%\begin{equation}
	A_{K}(\xi,t)=\sum_{k=0}^N a_k(\xi)\ t^k
$$	
%\end{equation}
	for some integer $N$, all $\xi\in S^{n-1}$, and all $t$ for which
	the set $K\cap \{x: \langle x,\xi\rangle =t\}$ is non-empty. Here, $a_k$ are functions on the sphere.

 Since the function $\xi \to A_K(\xi,t)$ is continuous, all the coefficients $a_k(\xi)$ are continuous functions on $S^{n-1}.$
Without loss of generality we can assume that the origin is an interior point of $K$, since polynomial integrability is invariant under translations. Observe that for all $k>N$ and all $\xi\in S^{n-1}$ we have 	$$\frac{\partial^k}{\partial t^k} A_K(\xi,t)\Big|_{t=0}=0.$$
We will use this  to conclude that $K$ is an ellipsoid in odd dimensions. First let us show that in the case of centrally symmetric bodies we need much less information.

\begin{theorem}\label{or-sym}
	Let $K$ be an infinitely smooth origin-symmetric convex body in $\mathbb R^n$, where $n$ is odd. Suppose that for some   even integer  $k> n$  and all $\xi\in S^{n-1}$ we have
	\begin{equation}\label{der_k}\frac{\partial^k}{\partial t^k} A_K(\xi,t)\Big|_{t=0}=0\end{equation}
	and
	\begin{equation}\label{der_k+2}\frac{\partial^{k+2}}{\partial t^{k+2}} A_K(\xi,t)\Big|_{t=0}=0.
	\end{equation}
	Then $K$ is an ellipsoid.
\end{theorem}

\begin{proof}
	
	It is known (see \cite[Thm 3.18]{K}) that the derivatives of   $A_{K}(\xi,t)$ with respect to $t$ at $t=0$ can be expressed in terms of the Fourier transform of powers of the Minkowski functional of $K$. Namely, if $k \ge 0$ is an even integer, $k\ne n-1$, then
	\begin{equation}\label{A^(k)}
	\frac{\partial^k}{\partial t^k} A_K(\xi,t)\Big|_{t=0} =    \frac{(-1)^{k/2}}{\pi(n-k-1)} \Big(\Vert x\Vert_K^{-n+1+k}
	\Big)^\wedge (\xi),\qquad \forall \xi\in S^{n-1}.
	\end{equation}

	Using condition (\ref{der_k}) and homogeneity of the Fourier transform of $  \|x\|_K^{-n+1+k}$, we get $$\Big(\Vert x\Vert_K^{-n+1+k}
	\Big)^\wedge (\xi)=0,\qquad \forall \xi\in \mathbb R^n\setminus\{0\}.$$
	It is well-known fact that a distribution supported at the origin is a linear combination of derivatives of the delta function
	(see, for example, \cite[Thm. 6.25]{R}). Therefore,  the Fourier transform of $  \|x\|_K^{-n+1+k}$ is a finite linear combination of derivatives
	of the delta function, implying that $ \|x\|_K^{-n+1+k}$ is a
	polynomial. Denoting $m=-n+1+k$, we have $$\|x\|_K^{m}=P(x),$$ for some homogeneous polynomial $P$ of even degree $m$. Similarly, (\ref{der_k+2}) implies
	$$\|x\|_K^{m+2}=Q(x),$$ where $Q$ is  homogeneous polynomial  degree $m+2$.
	
	The latter two equations yield $(P(x))^{m+2} = (Q(x))^{m}$ for all $x$.  Now   consider any two-dimensional subspace $H$ of $\mathbb R^n$. The restrictions of $P$ and $Q$ to $H$ are again homogeneous polynomials of degrees $m$ and $m+2$ correspondingly. Abusing notation, we will denote these restrictions by $P(u,v)$ and $Q(u,v)$, where $(u,v)\in \mathbb R^2$. Thus we have $ (P(u,v))^{m+2} = (Q(u,v))^{m}$ for all $(u,v)\in \mathbb R^2$.  Since both $P$ and $Q$ are homogeneous, the latter is equivalent to
	$$ (P(u,1))^{m+2} = (Q(u,1))^{m}, \qquad \forall u\in \mathbb R.$$
	We have the equality of two polynomials of the real variable $u$,
	therefore these polynomials are equal for all $u \in \mathbb C$. Let $u_0$ be a complex root of $P(u,1)$ of multiplicity $\alpha\le m$. Then $u_0$ is also a root of $Q(u,1)$ of some multiplicity $\beta\le m+2$. Hence we have
	$$\alpha (m+2) = \beta m,$$
	$$\frac{\alpha}{\beta} = \frac{m}{m+2}.$$
	Recall that $m$ is even, say $m = 2 l$, $l \in \mathbb N$.
	Thus $$\frac{\alpha}{\beta} = \frac{l}{l+1}.$$
	Since $l$ and $l+1$ are co-prime, there are only two possibilities for $\alpha$ and $\beta$: either $\alpha=l$, $\beta = l+1$, or $\alpha = 2l$, $\beta = 2\l + 2$. The latter is impossible since it implies that
	$$ \|(u,v) \|_{L\cap H}^m = P(u,v) = c (u - v u_0)^m,$$
	for some constant $c$.
	So the remaining possibility is that $P(u,1)$ has two complex roots, say $a$ and $b$ of multiplicity $l$. Therefore,
	\begin{eqnarray*}\|(u,1) \|_{L\cap H}^m&=&P(u,1)\\ & =& c [(u -   a)(u -
		b)]^l = c [u^2 - (a+b) u   + ab ]^l \\ & =&  c [ u^{2l} - l(a+b) u ^{2l -1}
		+ \left( {l\choose 2} (a+b)^2+ lab\right) u^{2l -2} + \cdots].
	\end{eqnarray*}
	Since the restriction of this polynomial to $\mathbb R$ has real coefficients, it follows that $a+b$ and $ab$ are real numbers. Since $a$ and $b$ cannot be real, we conclude that they are complex conjugates of each other. Therefore,    $\|(u,v) \|_{K\cap H}^2 =  \bar c [u^2 - (a+b) u v  + ab v^2 ]$ is a nondegenerate quadratic form. Thus $L\cap H$ is an ellipse. Since every 2-dimensional central section of $L$ is an ellipse, $L$ has to be an ellipsoid. The latter is a consequence of the Jordan - von Neumann characterization of inner product spaces by the parallelogram equality;
	see \cite{JN}.
\end{proof}

   Proving this result for non-symmetric bodies  is more involved, so we will just provide a sketch of the proof.

\begin{theorem} \label{odd-dim}  Let $n$ be a positive odd integer, and let $K$ be an infinitely smooth convex body in
	$\mathbb R^n,$ containing the origin in its interior.
	Suppose there exists    $N\ge n$ such that for every integer $k\ge N$ and every $\xi\in S^{n-1}$  we have
	$$\frac{\partial^k}{\partial t^k} A_K(\xi,t)\Big|_{t=0}=0.$$
	Then   $K$ is an ellipsoid.
	
\end{theorem}

\begin{proof}
	We will use an analog of formula (\ref{A^(k)}) for non-symmetric bodies  obtained in \cite{RY}. 	
	  If $k \ge 0$ is an even integer, $k\ne n-1$, then for every $\xi\in S^{n-1}$,
	\begin{equation}\label{A^(k-even)}
\frac{\partial^k}{\partial t^k} A_K(\xi,t)\Big|_{t=0} =    \frac{(-1)^{k/2}}{2\pi(n-k-1)} \Big(\Vert x\Vert_K^{-n+1+k}+\Vert - x\Vert_K^{-n+1+k}
	\Big)^\wedge (\xi),
	\end{equation}
	and if $k> 0$ is an odd integer, $k\ne n-1$, then
	\begin{equation}\label{A^(k-odd)}
	\frac{\partial^k}{\partial t^k} A_K(\xi,t)\Big|_{t=0} =    \frac{i(-1)^{(k-1)/2}}{2\pi(n-k-1)} \Big(\Vert x\Vert_K^{-n+1+k} - \Vert - x\Vert_K^{-n+1+k}
	\Big)^\wedge (\xi).
	\end{equation}
	
	Setting the Fourier transforms equal to zero in (\ref{A^(k-even)}) and (\ref{A^(k-odd)}), and arguing as in the proof of Theorem \ref{or-sym}, we get that 	
	  $\|x\|_K^{m}+\| - x\|_K^{m}$ is a polynomial for every even $m\ge N-n+1$, and  $\|x\|_K^{m} -\| - x\|_K^{m}$ is a polynomial for    every odd $m\ge N-n+1$.

	Thus for any integer $s\ge (N-n+1)/2$ we have

	$$\|x\|_K^{2s+1} - \|-x\|_K^{2s+1} = P(x)$$
	and
	$$\|x\|_K^{4s+2} + \|-x\|_K^{4s+2} = Q(x),$$
	where $P$ and $Q$ are homogeneous polynomials of degrees    $2s+1$ and  $4s+2$ respectively.

Solving the latter system of equations, we get, for every   $s\ge (N-n+1)/2$, that
	$$ \|x\|_K^{2s+1} = \frac12\left(P_s(x)+\sqrt{Q_s(x)}\right),$$
	where $P_s$ is an odd homogeneous polynomial of degree $2s+1$ and $Q_s$ is an even homogeneous polynomial of degree $4s+2$.

 Theorem 3.6 from \cite{KMY} allows to conclude   that the Minkowski functional of $K$ is of the form
	$$ \|x\|_K = R(x) + \sqrt{S(x)},$$
	where $R$ is a linear polynomial   and $S$ is a positive quadratic  polynomial. From this it is easy to see that $K$ is an ellipsoid. Indeed, if $x\in \partial K$, then $\|x\|_K=1$ and therefore
	$$ 1- R(x) = \sqrt{S(x)}.$$
	Squaring both sides, we get an equation of a quadric surface. Since $K$ is compact, this surface can only be the surface of an ellipsoid.	
\end{proof}

The methods used in this section also allow us to obtain an alternative proof of the   result obtained in \cite{Ag}, saying that there are no infinitely smooth polynomially integrable convex bodies in $\mathbb R^n$ for even $n$ (see Corollary \ref{cor}). We will prove a little more.

\begin{theorem} \label{even-dim}
	Let $n$ be a positive even integer. There is no infinitely smooth  convex body $K\subset \mathbb R^n$ containing the origin in its interior and satisfying
	\begin{equation}\label{m-der}\frac{\partial^m}{\partial t^m} A_K(\xi,t)\Big|_{t=0} = 0
	\end{equation}
	for some even $m\ge n$ and all $\xi\in S^{n-1}.$
\end{theorem}
\begin{proof}
Assume that there exists  an infinitely smooth   convex body  $K$ in $\mathbb R^n$
	satisfying  (\ref{m-der}) for some $m\ge n$. Let $m$ be even.
 Using (\ref{A^(k-even)}) we get
	$$
	\Big(\Vert x\Vert_K^{-n+1+m}+\Vert - x\Vert_K^{-n+1+m}
	\Big)^\wedge (\xi)=   { (-1)^{m/2}}{2\pi(n-m-1)}
\frac{\partial^m}{\partial t^m} A_K(\xi,t)\Big|_{t=0} = 0,
	$$
	for every $\xi\in S^{n-1}$.

	Thus the Fourier transform of $ f(x)= \|x\|_K^{-n+1+m}+\|-x\|_K^{-n+1+m}$ is zero outside of
	the origin, implying that $  f(x)$ can only be a
	polynomial. This polynomial has to be even, since the
	function $  f(x)$  is even. On the other hand, since
	$-n+1+m$ is an odd number, $  f(x)$  has to be an odd
	polynomial. Thus $f(x)$   is zero for all $x\in \mathbb
	R^n$, which is impossible.
\end{proof}

We have just proved that the section function $A_K(\xi,t)$ is never a polynomial with respect to $t$ in even dimensions. However, ellipsoids in even-dimensional spaces have the section function which in a sense is close to a polynomial, namely, this function differs from a polynomial by a simple factor. Indeed, if $K$ is an ellipsoid centered at the origin, then $A_K(\xi,t)=C(\xi)(h^2_K(\xi)-t^2)^{\frac{n-1}{2}}$ where $h_K(\xi)$ is the support function. It follows that if $n$ is even then
$A_K(\xi,t)$ can be represented in two ways:
$$A_K(\xi,t)=C(\xi) \sqrt{ h_K^2(\xi)-t^2}\ P(\xi,t)=C(\xi)\frac{ P_1(\xi,t)}{ \sqrt{h_K^2(\xi)-t^2} } ,$$
where $P(\xi,t), \ P_1(\xi,t)$ are polynomials in $t.$ It was proved in \cite{AKRY} that such a presentation of the section function characterizes ellipsoids in even-dimensional spaces.

Denote by $\mathcal{H}$ the Hilbert transform
\begin{equation}\label{E:H}
\mathcal{H} f(t)= \frac{1}{\pi} p.v. \int\limits_{\mathbb R}  \frac{f(s)}{t-s} ds
\end{equation}
of a continuous function $f$ with sufficiently fast decay at infinity.

The main result of the article \cite{AKRY} is as follows.
\begin{theorem} \label{T:Main-even} Let $n$ be an even positive integer. Let $K$ be a bounded convex domain in $\mathbb R^n$ with $C^{\infty}$ boundary $\partial K.$
The following are equivalent:
\begin{enumerate}[(i)]
\item The section function $A_K(\xi,t)$ has the form
$$A_K(\xi,t)=\sqrt{q(\xi,t)}\ P(\xi,t),$$
where $P(\xi,t), \ q(\xi,t)$ are {continuous in $\xi$} and polynomials in $t$ with $\deg q(\xi,\cdot) =2 ; \ q(\xi,t) > 0.$
\item The section function $A(\xi,t)$ has the form
$$A_K(\xi,t)=\frac{P(\xi,t)}{\sqrt{q(\xi,t)}},$$
where $P(\xi,t), \ q(\xi,t)$ are as in (i).
\item  For every fixed $\xi\in S^{n-1},$ the Hilbert transform of the function $t\to A_K(\xi,t)$ is a polynomial of $t.$
\item $K$ is an ellipsoid.
\end{enumerate}
The equalities for the section function appearing above hold for values of $t$ for which the hyperplane $\xi \cdot x=t$ meets $K.$ 

%It is assumed, that the above representations for $A_K(\xi,t)$ are  true so long as the hyperplane $x \cdot \xi=t$ meets $K.$

\end{theorem}

Note that the appearance of the Hilbert transform in the latter theorem is not very surprising, since $A_K(\xi,t)$ is the Radon transform of the indicator of the body $K,$ and the Hilbert transform  is involved in the back-projection inversion formula for the Radon transform in even-dimensional spaces.

\bigskip

\section{Domains with algebraic Radon transform without real singularities } \label{S:S4}

In this section we consider classes of bodies with algebraic properties more general than polynomial integrability.
 These classes correspond to a choice of the form of the corresponding defining polynomial $\Psi$ in Definition \ref{D:algRT}.

Classes of algebraic functions are characterized by the form of the corresponding defining polynomial $\Psi$ in Definition \ref{D:algRT}. In the case of polynomially and rationally integrable domains (see Introduction), the defining polynomial $\Psi(\xi,t,w)$  is linear  with respect to $w$ and hence does not have multiple roots $w.$ Starting from this observation, we consider equations $Q(\xi,t,w)=0$  having only simple roots $w \in \mathbb C,$ for  any fixed $\xi \in S^{n-1}$ and for any real $t.$  This means that for any fixed $\xi,$ the algebraic function $w=w(t)$ has no real branching points. The set $Br_{\xi}$ of branching points is finite for every $\xi,$ but depends on $\xi.$ We will also assume that the union $\cup_{\xi \in S^{n-1}} Br_{\xi}$ of all branching points when $\xi$ runs over the unit sphere. Also it is a bounded set in the complex plane $\mathbb C,$ i.e., the branching points do not go to infinity when the normal vector $\xi$ runs over the sphere $S^{n-1}.$

We will formulate the conditions for the defining polynomial $\Psi(\xi,t,w)$ in terms of its discriminant with respect to the variable $w:$
$$D(\xi,t)=\psi_N(\xi,t)^{2N-1} \prod_{i<j}\big(w_i(\xi,t)-w_j(\xi,t)\big),$$
where $N=deg_{w} \Psi, \ \psi_N(\xi,t)$ is the leading coefficient of $\Psi(\xi,t,w)$ as a polynomial of $w,$ and $w=w_i(\xi,t)$ are the roots
of the algebraic equation $\Psi(\xi,t,w)=0.$ The discriminant is a polynomial of the coefficients of the polynomial $w \to \Psi(\xi,t,w)$ and since $\Psi(\xi,t,w)$ is a polynomial in $t,$ the function $D(\xi,t)$ is a polynomial with respect to $t.$

\begin{definition} \label{D:free} We say that the body $K$ has {\bf algebraic Radon transform without real singularities} if
the discriminant $D(\xi,t) \neq 0$ for all $\xi \in S^{n-1}$ and for all real $t,$ and also the leading coefficient of the polynomial $t \to D(\xi,t)$ does not vanish for all $\xi \in S^{n-1}.$ The latter  means that the null-set $\{ \xi \in S^{n-1}, \ z \in \mathbb C: D(\xi,z)=0 \}$ is a compact subset of $S^{n-1} \times \mathbb C.$
\end{definition}

Let us explain the relation of the condition $d(\xi) \neq 0, \ \xi \in S^{n-1}$ in Definition \ref{D:free} with the location of branching points of the algebraic function $t \to w(\xi,t).$ Since $\Psi(\xi,t,w)$ is continuous with respect to  $\xi,$ then
$D(\xi,t)$ and $d(\xi,t)$ are the same. Therefore if $d(\xi)$ does not vanish on $S^{n-1}$  then $|d(\xi)| \geq C > 0, \ |\xi|=1.$
Write
$$D(\xi,t)=d(\xi)t^M+d_1(\xi)t^{M-1}+ \cdots+ d_M(\xi),$$
where $t$ is complex. Then
$$|D(\xi,t)| \geq |t|^M \left(C -\frac{|d_1(\xi)|}{|t|} - \cdots - \frac{|d_M(\xi)|}{|t|^M} \right) .$$
Since the coefficients $d_1(\xi), \cdots , d_M(\xi)$ are continuous and hence bounded on $S^{n-1}$ then there exists $R>0$ such that $D(\xi,t) \neq 0$ for all
$\xi \in S^{n-1}$ and for all complex $t$ with $|t|>R.$

Now, let $|t_0| >  R.$ Then  $D(\xi, t_0) \neq 0$ for any  $\xi \in S^{n-1}.$ Fix such $\xi \in S^{n-1}$. Then, by definition of the discriminant, all the roots of the polynomial $w \to  Q(\xi,t,w)$ are simple when $t$ is close  to $t_0$   and, moreover, they are continuous functions of $t$ in a neighborhood of $t_0.$
 Therefore, $t_0$ is a regular point of the algebraic function $t \to w(\xi,t),$ i.e., $ t \notin Br_{\xi}.$ Thus, the union of the sets $Br_{\xi}, \ \xi \in S^{n-1},$
is contained in the disc $|t| \leq R$ and  is bounded. Therefore, all the branching points $(\xi,t)$ are contained in a compact subset of $S^{n-1} \times \mathbb C.$
\smallbreak
\noindent{\bf Examples.} If $A_K(\xi,t)$ is a polynomial in $t$ then the
 Radon transform of the characteristic function of any polynomially integrable body $K$ has no real singularities since if $A_K(\xi,t)=P(\xi,t)$ is a polynomial in $t$ then $D(\xi,t)=1.$ Another example is any rationally-integrable body $K$ with $A_K(\xi,t)=\frac{P(\xi,t)}{Q(\xi,t)}$ and $Q(\xi,t) \neq 0$ for $(\xi,t) \in S^{n-1} \times \mathbb R.$ In this case $D(\xi,t)=Q(\xi,t)$ satisfies the above condition.

\begin{theorem} [\cite{Ag1}]\label{T:Thm1} Let $n$ be an odd integer, and let $K$ be a body in $\mathbb R^n$ with $C^{\infty}$ boundary $\partial K.$ Suppose that the Radon transform $A_K(\xi,t)$ of $\chi_K$ is an algebraic function, free of real singularities (Definition \ref{D:free}). Then $\partial K$ is an ellipsoid.
There are no bodies satisfying all those properties if $n$ is even.
\end{theorem}

\begin{remark} \rm In \cite{Ag1}, Theorem \ref{T:Thm1} is formulated in terms of the cutoff function $V_K(\xi,t).$
\end{remark}

\subsection{Outline of the proof}

The idea of the proof is to show that $K$ is  polynomially integrable and to use the result of \cite{KMY}.

First, we notice that the condition that the polynomial $t \to \Psi(\xi,t,w)$ has no real multiple roots implies that for any fixed $\xi$ the algebraic function $w=w(t)$ defined by the equation $\Psi=0$ has no branching points for real $t.$ Hence, for every fixed $\xi \in S^{n-1},$ the function $t \to A_K(\xi,t)$ (where $t$ is in an open  interval such that $\{ \langle \xi, x \rangle =t \}$ meets $K$) represents a germ of a real-analytic branch of the algebraic function $w=w(t)$ and
extends to all $t \in \mathbb R$  as a real-analytic function. However, this situation is impossible when $n$ is odd, since the function $A_K(\xi,t)$ does not extend analytically through the tangent plane at boundary Morse points $a \in \partial K:$
\begin{lemma} \label{L:bndry-ext}
Let $a \in \partial K$ be a Morse point and let $\langle \xi_0,x \rangle =t_0$ be the tangent plane to $\partial K$ at the point $a.$
Then $A_K(\xi_0,t)= c (t-t_0)^{\frac{n-1}{2}} \big(1+o(1) \big),  \ t \to t_0, \ c \neq 0.$
\end{lemma}
\begin{corollary} [\cite{Ag}] \label{cor} Let $n$ be even. There is no polynomially integrable body $K \subset \mathbb R^n$ with infinitely smooth boundary.
\end{corollary}
Indeed, if $n$ is even then the exponent $\frac{n-1}{2}$ is fractional and since $\partial K$ contains an open set of Morse points, $A_K(\xi,t)$ cannot be a polynomial with respect to $t$ for all $\xi \in S^{n-1}.$  In Theorem \ref{even-dim} a different approach to this phenomenon is presented. Notice, that both Corollary \ref{cor} and Theorem \ref{even-dim}  are consonant with Theorem \ref{even} which states that there is no infinitely smooth algebraically integrable bodies in even-dimensional spaces. 

The proof of Lemma \ref{L:bndry-ext} immediately follows from the fact that, in a neighborhood of the Morse point $a$ the $C^{\infty}$ hypersurface  $\partial K$ can be represented (after a suitable rotation) as the graph of the function $x_n=a_n+ \sum\limits_{j=0}^{n-1} \lambda_j^2 x_j^2+ o(x_1^2+ \cdots +x_{n-1}^2).$ Then $\xi=(0,...,0,1),$ the tangent plane is given by $x_n=a_n=t_0$ and $A_K(\xi,t)$ is the volume of the cross-section $\{x_n=t\} \cap K.$ It is  equal, up to a small term of higher order, to the volume of the ellipsoid
$\sum\limits_{j=0}^{n-1}\lambda_j ^2 x_j^2 =t-t_0$ which is proportional to $(t-t_0)^{\frac{n-1}{2}}.$
Again, if $n$ is even then $\frac{n-1}{2}$ is non-integer and hence the function $A_K(\xi,a)$ is not real-analytic in $t$ near $t=t_0.$ It remains to notice that on any smooth closed hypersurface there is an open set of Morse points. This proves that if $n$ is even then for no body $K$ the conditions of Theorem \ref{T:Thm1} are fulfilled.

From now on, $n$ is odd. Applying a translation, we can assume that $0$ is an interior point of the body $K.$ Then for any $\xi \in S^{n-1}$ the function $t \to A_K(\xi,t)$ is well defined and real-analytic in an interval $|t|<\varepsilon.$  The key fact is that the Fourier coefficients of $A_K(\xi,t)$ with respect to $\xi$ are polynomials of the variable $t$:
\begin{lemma}\label{L:poly}  Let $n$ be odd. Let $A_K(\xi,t)=\sum\limits_{k=0}^{\infty}\sum\limits_{\alpha=1}^{d_k} p_{k,\alpha}(t)Y_k^{(\alpha)}(\xi)$ be the Fourier decomposition of the function $\xi \to A_K(\xi,t), \ |t| < \varepsilon$ on the unit sphere. Here $\{Y_k^{\alpha}\}_{\alpha=1}^{d_k}$ is an orthonormal basis in the space $H_k$ of all spherical harmonics of degree $k.$  Then the Fourier coefficients $p_{k,\alpha}(t)$ are polynomials and  $\deg p_{k,\alpha} \leq k+n.$
\end{lemma}

\begin{proof}  Write the back-projection inversion formula for the Radon transform $A_K(\xi,t)=(\R \chi_K)(\xi,t)$ (\cite[Chapter 1, Theorem 3.1]{He}) in  odd-dimensional Euclidean spaces:
\begin{equation}\label{E:inversion}
1=\chi_K(x)=c \int_{|\xi|=1} \frac{d^{n-1}}{dt^{n-1}} A_K(\xi, \langle x, \xi\rangle  )d\xi,
\end{equation}
$d\xi$ is the normalized Lebesgue measure on $S^{n-1},$ when $x \in K.$
Applying the Laplace operator to the Radon transform results in differentiating twice in $t,$ (\cite[Chapter 1, Lemma 2.1]{He}), hence applying the Laplace operator to the both sides yields
$$\int_{|\xi|=1} B(\xi, \langle x, \xi \rangle)d\xi=0,$$
where we have denoted $B(\xi,t)=\frac{d^{n+1}}{dt^{n+1}}A_K(\xi, t ).$

The function $t \to B(\xi,t)$ is real analytic near $t=0:$
$$B(\xi,t)=\sum_{j=0}^{\infty}b_j(\xi)t^j, |t| <\varepsilon,$$
and hence
$$\sum\limits_{j=0}^{\infty} \int_{|\xi|=1} b_j(\xi)\langle x, \xi \rangle ^j d\xi=0$$
for $|x|<\varepsilon.$
Comparing homogeneous (in $x$) terms we get
$$\int_{|\xi|=1} b_j(\xi) \langle x, \xi \rangle ^j d\xi=0.$$

The functions $\mathbb R^n \ni \xi \to  \langle x, \xi \rangle ^{j},$ when $x$ runs over an open neighborhood of $x=0,$ span the space $\mathcal P_{j}$ of all homogeneous polynomials of degree $j.$     Since the restricted space $\mathcal P_{j} \vert_{S^{n-1}}= \bigoplus_{s=0}^{\frac{j}{2}} H_{j-2s},$ we conclude that each coefficient $b_j(\xi)$ is orthogonal on $S^{n-1}$ to all spherical harmonics of degree $\leq j,$  of the same parity with $j.$

Besides, $A_K(\xi,t)=A_K(-\xi,-t)$ implies that $a_j$ is even when $j$ is even and odd when $j$ is odd. This property is inherited by the functions $b_j(t)=a_j^{(n+1)}(t)$ because $n+1$ is even. Thus, $b_j$ is orthogonal to all spherical harmonics of the parity which is opposite to that of $j.$ Then $b_j$ is orthogonal on $S^{n-1}$ to all spherical harmonics $Y_k^{(\alpha)}$ of degree $1 \leq k \leq j,$ regardless of the parity of $j,$ and hence
$$b_j=\sum\limits_{k \geq j+1}\sum\limits_{\alpha=1}^{d_k} b_{j,k,\alpha} Y_k^{\alpha}.$$
Then
$$B(\xi,t)=\sum_{j} \sum\limits_{k \geq j+1}\sum_{\alpha=1}^{d_k} b_{j,k,\alpha} Y_k^{\alpha}(\xi)t^j
=\sum\limits_{k}\sum_{\alpha=1}^{d_k} q_{k,\alpha}(t)Y_k^{\alpha}(\xi),$$
where
$$q_{k,\alpha}(t)=\sum\limits_{j \leq k-1} b_{j,k,\alpha}t^j $$
is a polynomial of degree $\leq k-1.$
By the construction,  $q_{k,\alpha}(t)=p^{(n+1)}_{k,\alpha}(t)$ hence $p_{k,\alpha}(t)$ is also a polynomial, of degree $\leq k-1+n+1=k+n.$
\end{proof}

The condition for the discriminant of the polynomial $\Psi$ can be rephrased as follows: the projection of the set $\Psi(\xi,t,w)=\partial_w \Psi(\xi,t,w)=0$ on $t$ is a compact subset of $\mathbb C$ disjoint from the real axis. Therefore, if $R>0$ is sufficiently large, then the circle $C_R=\{|z|=R\}$  encloses all complex $t$ such that for some $\xi \in S^{n-1}$ the polynomial $w \to \Psi(\xi,t,w)$ has a multiple root.

Denote $C_R^{\pm}= [-R,R] \cup (C_R \cap \Pi^{\pm}),$ where $\Pi^{\pm}$ stands for,  correspondingly, upper and lower halfplanes. Since the fundamental group
of $S^{n-1} \times (C_R^+ \setminus 0)$ is trivial, and all $w$-zeros of $\Psi(\xi,t,w)=0, \ (\xi,t) \in S^{n-1} \times C_R^{+}$ are simple, Lemma about covering homotopy (\cite[Theorem 16.2]{Hu})   implies that there is a leaf $w^+_{\xi}(t)$ of the multi-valued algebraic function $w=w_{\xi}(t)$ defined by $\Psi(\xi,t,w),$ which is continuous on $S^{n-1} \times  (C_R^{+} \setminus 0)$ and coincides with $A_K(\xi,t)$ on $S^{n-1} \times \big( (-\varepsilon, \varepsilon) \setminus 0 \big)$ (and therefore is continuous on $S^{n-1} \times C_R^+).$

Fix a spherical harmonic $Y_k^{\alpha}.$ The functions $b^+_{k,\alpha}(t)=\int_{|\xi|=1} Y_k^{\alpha} w_{\xi}^+ (t)d\xi$ are real analytic in $t$ because $w^+_{\xi}(t)$ is real analytic by the construction. On $(-\varepsilon, \varepsilon),$ $b_{k,\alpha}^+(t)$ has Fourier coefficients which are polynomials  by Lemma \ref{L:poly}. Therefore, $b_{k,\alpha}^+(t)$ are polynomials and by Cauchy's theorem
$$\int_{C_R^+}b^+_{k,\alpha}(t)dt=0.$$ It yields
$$\int_{|\xi|=1} \int_{t \in C_R^+} Y_k^{\alpha}(\xi)w^+_{\xi}(t) t^m dt d\xi=0$$
for all $m=0,1, \cdots.$
Since the harmonic in the integral is arbitrary, we have
$$\int_{C_r^+}w_{\xi}^+(t)t^m dt=0.$$
The vanishing  complex moments imply that $w_{\xi}^+$ is the boundary value of a function, analytic inside the closed contour $C_R^+.$ Similarly, constructing an analytic extension $w_{\xi}^{-}(t)$ along  $S^{n-1} \times C_R^{-}$ in the lower half-plane yields that this extension  is a boundary value of an analytic function inside the contour $C_R^-.$ Since there is no ramification points outside of the circle $C_R,$ we conclude that $A_K(\xi,t)$ extends to $S^{n-1} \times \mathbb C$ as an entire function in $t.$ By the Great  Picard theorem, entire algebraic functions are polynomials and therefore $A_K(\xi,t)$ is a polynomial in $t,$ i.e., $K$ is polynomially integrable. Then Theorem 1 from \cite{KMY} implies that $K$ is  an ellipsoid.

\section{Local polynomial integrability } \label{S:S5}

\subsection{Polynomial integrability and finite stationary phase expansion}
Recall that the stationary phase method of asymptotic expansion of oscillatory integrals depending on a large parameter is based on the idea that that main contribution in the asymptotic is delivered by critical points of the phase function.

Generally, the expansion is presented as an infinite asymptotic series.  However,  such an expansion can be finite, i.e., have only finite number of nonzero terms. This phenomenon is related to the so called Hamiltonian maps and was studied in \cite{AB},  \cite{Be}.
A simplest example of finite asymptotic expansion is given by the oscillatory integral on the unit sphere:
$$I(\lambda)=\int_{S^{n-1}} e^{i\lambda x_n} dS(x),$$
where the dimension $n$ is odd.
Indeed, integration by parts yields
$$I(\lambda)=c \int_{-1}^{1} e^{i\lambda x_n} (1-x_n^2)^{\frac{n-3}{2}} dx_n= \sum_{x_n=\pm 1} e^{i\lambda x_n} Q_{x_n}(\frac{1}{\lambda}),$$
where $Q_{\pm 1}$ are polynomials.

Surprisingly, the polynomial integrability of a body $K$ appears equivalent to the finiteness of the stationary phase expansion for oscillatory Fourier integrals on $\partial K,$ with linear phases:

\begin{proposition} \cite{AgStat} \label{P:stat} Let $K$ be a convex body in $\mathbb R^n,$ with $C^{\infty}$ boundary. Then $K$ is polynomially integrable if and only if the family of oscillatory integrals $I_{\xi}(\lambda), \ |\xi|=1,$
$$I_{\xi}(\lambda)=\int_{\partial K} e^{ i \lambda \langle x, \xi \rangle } \langle \xi, n(x) \rangle dS(x),$$
where $n(x)$ is the unit outward normal vector,  and $dS$ is the area measure on $\partial K,$ possesses a finite stationary phase expansion of the form
\begin{equation} \label{E:I}
I_{\xi}(\lambda) =\sum_{\pm} e^{i\lambda b_{\pm}(\xi)}  Q_{\xi, \pm}(\frac{1}{\lambda}),
\end{equation}
where $  b_{-}(\xi)  = min_{x \in  K} \langle x, \xi \rangle, \   b_+(\xi)  = max_{x \in  K} \langle x, \xi \rangle.$
\end{proposition}

\begin{proof}  By the Projection-Slice Theorem (\cite[Chapter 1, Section 2, Formula (4)]{He}) the Fourier transform of $\chi_K$ equals
$$ \widehat \chi_K(\lambda \xi)=\int_{b_{-}(\xi)}^{b_+(\xi)} e^{i \lambda t} A_K(\xi,t)dt.$$
If $A_K(\xi,t)$ is a polynomial in $t$ then integration by parts shows that $\widehat \chi_K (\lambda \xi)$ is represented in the form as in the right hand side in (\ref{E:I}).
On the other hand, by Stokes' formula:
$$
\begin{aligned}
&\lambda\, \widehat \chi_K(\lambda \xi)=- \lambda^{-1} \int_K \Delta e^{i \langle x, \lambda \xi \rangle} dx
= -\lambda^{-1} \int_{\partial K} \frac{\partial}{\partial n}e^{i \langle x, \lambda \xi \rangle}dS(x) \\
&=- i \int_{\partial K} e^{i \lambda \langle x, \xi \rangle} \langle \xi, n(x) \rangle dS(x)=-i  I_{\xi}(\lambda).
\end{aligned}
$$

It shows that $I_{\xi}(\lambda)$ has the required form (\ref{E:I}).  The "only if" part is proved by a ``reverse"  reasoning.
\end{proof}
\smallbreak
\begin{remark} \rm
This Proposition gives another argument why there are no polynomially integrable bodies in even dimension.
Indeed, if $\xi$ is such that at least one of the points $b_{\pm}(\xi)$ with the normal vector $\pm \xi$ has nonzero Gaussian curvature then the leading term of the expansion is known to be $c \lambda ^{-\frac{n-1}{2}}$ and if $n$ is even then the expansion cannot have the form (\ref{E:I}) because the exponent
$\frac{n-1}{2}$ is non-integer.
\end{remark}

\subsection{Characterization of locally polynomially integrable surfaces}
The relation with stationary phase expansion can be exploited for the study of a local version of polynomial integrability property.
Given a smooth strictly convex hypersurface $M \subset \mathbb R^n,$ denote by $T_a(M)$ the tangent plane at the point $a \in M$ and by $\nu_a$ the unit normal vector at $a$ directed ``inside" $M,$ i.e., $M$ lies in the half-space $T_a + \mathbb R_{+} \nu_a.$

\begin{definition} Let $M$ be a smooth strictly convex hypersurface in $\mathbb R^n.$ We say that $M$ is locally polynomially integrable if the $(n-1)$-dimensional volume
$A(a,t)$ of the $(n-1)$-dimensional domain in $T_a(M)+t\nu_a,$ bounded by $M \cap (T_a(M) + t\nu_a),$ where $t>0$ is sufficiently small, polynomially depends on $t.$
\end{definition}
\noindent{\bf Examples}. All strictly convex quadrics in $\mathbb R^{2k+1},$ i.e., ellipsoids, elliptic parabo\-loids,  single sheets of   two-sheet hyperboloids,  are locally polynomially integrable.
\smallbreak
The same argument as in the case of global polynomial integrability shows that there are no locally polynomially integrable surfaces in even dimensions.
\begin{conjecture}\label{conj} All smooth polynomially integrable hypersurfaces are (strictly convex) quadrics in $\mathbb R^{2k+1}.$
\end{conjecture}
Expressing polynomial integrability in terms of the stationary phase expansion appeared useful in proving results toward this Conjecture.
Lemma \ref{L:local-stat}, a local version of Proposition \ref{P:stat}, expresses the function $v(a,t)$ in terms of an asymptotic expansion of an oscillatory integral in a neighborhood of $a$
with a linear phase with the only critical point $a.$  Since the construction is local, we have to use a cut-off function. As well as in Proposition \ref{P:stat}, the
essential part of the expansion is finite, however the representation now is not exact due to the presence of a reminder of fast decay with respect to the large parameter $\lambda.$

For a fixed $a \in M$ consider the oscillatory integral
\begin{equation*}
I_{a}(\lambda):=\int\limits_M \frac{\partial}{\partial \nu_x} \Big[ e^{i \lambda  \langle x, \nu_a \rangle } \Big]
\rho_a(\langle x, \nu_x \rangle) dS(x),
\end{equation*}
where $\rho_a(u)$ is an infinitely smooth function of one variable,  supported in a sufficiently small interval $(-\varepsilon_a, \varepsilon_a), \ \rho_a(0)=1.$
The point $a$ is a critical point of the phase function and there are no other critical points if $\varepsilon$ is small.
\begin{lemma}\label{L:local-stat} The  volume function $A(a,t)$ is a polynomial in $t$ if and only if
\begin{equation}\label{E:local-stat}
I_{a}(\lambda)=c e^{i \langle a, \nu_a \rangle } Q_{a} \Big(\frac{1}{\lambda} \Big)+ o(\frac{1}{\lambda^m}), \lambda \to \infty,
\end{equation}
where $Q_a$ is a polynomial and $ m \in \mathbb N$ is arbitrary.
\end{lemma}

Recall that a point $a \in M$ is {\it elliptic} if all the principal curvatures  at $a$ are positive. Near such a point, $M$ can be represented, after a suitable translation and rotation, as the graph $x_{n}=f(x_1,...,x_{n-1})$ where the second differential is of the form $d^2_af (h) = \sum_{j=1}^{n-1}\gamma_j h_j^2, \ \gamma_j>0.$
We say that $a$ is an {\it elliptic point of contact of order $ q$} if $d^3f_a=...=d^{q}f_a=0.$

The following theorem is a partial result towards Conjecture \ref{conj}.
\begin{theorem} \label{T:local}\cite{AgStat}  Let $M$ be a real analytic locally polynomially integrable  hypersurface in $\mathbb R^n, n=2k+1.$ Suppose that $M$ contains an elliptic point $a_0 \in M$
of contact of order $q > 4.$ Then $M$ is an elliptic paraboloid.
\end{theorem}
Ellipsoids and two-sheet hyperboloids do not satisfy the condition about an elliptic point of a high order and hence are out of consideration in Theorem \ref{T:local}. Getting rid of this condition (as well as the condition of real analyticity) would allow to obtain a full description of locally polynomially hypersurfaces.

The proof is based on a variation of the expansion in Lemma \ref{L:local-stat} with respect to $a,$ in a neighborhood of the elliptic point $a_0.$ More precisely, we represent $M$ near the point $a_0$ as the graph $x_n=f(x_1,...,x_{n-1})$ and then pass to Morse coordinates $u$ in which $f$ has the form $f(u)=u_1^2+...+u^2_{n-1}.$ Then we come  up with an oscillatory integral with a quadratic phase function. We parametrize the  integral $I_a(\lambda)$ by the normal vector $\xi=\nu_a$ and then differentiate in $\xi,$ applying, at the point $a_0,$ to (\ref{E:local-stat}) the Schrodinger operator $\Box=i\lambda \frac{\partial}{\partial \xi_n} - \Delta^{\prime},$ where $\Delta^{\prime}$ is the Laplace operator  in $\xi_1, ..., \xi_{n-1}.$ The condition of the high order of tangency at $a_0$  yields that this operation results in reducing the degree of the polynomial $Q$ in (\ref{E:local-stat}) and shortening the length of the expansion. Then, by applying the iterated operator $\Box^{k}$ with sufficiently large $k,$ we annihilate the essential part of the expansion. Furthermore, the coefficients of the stationary phase expansion with quadratic phase are expressed as powers of the Laplace operator of the density function at the point $a$ and in our case all of them are zero. Being translated in terms of the Morse diffeomorphism $u=u(x),$ vanishing powers of the Laplace operator implies that the mapping $u(x)$ is affine. In turn, this implies that $f$ is  a quadratic polynomial and, correspondingly, $M$ is an elliptic paraboloid.

\section{Domains  with  algebraic $X$-ray transform} \label{S:S6}
As it was mentioned before, polynomial integrability never occurs in even dimensions. However, the Radon transforms of the indicator functions of ellipsoids in even dimensions are, in a sense, close to being polynomials. Namely, the sectional volume function $A_E(\xi,t),$ where $E$ is an ellipsoid, is a {\it square root of a polynomial} in $t.$  Thus, both in even and odd dimensions, $A^2_E(\xi,t)$ is a polynomial.

\begin{conjecture} \label{c61} Let $K$ be a body in $\mathbb R^n$ with $C^{\infty}$ boundary. Suppose that there exists $m \in \mathbb N$ such that $A_K^{m}(\xi,t)$ is a polynomial in $t.$ Then $\partial K$ is an ellipsoid and therefore we can take $m=1$ if $n$ is odd and $m=2$ if $n$ is even.
\end{conjecture}
The following theorem confirms Conjecture  {\ref{c61}} for $n=2$ and for domains $K$ with algebraic boundary $\partial K$. This means that  $\partial K$ is contained in the zero set of a nonzero polynomial. In dimension 2 the function $A_K(\xi,t)$ is the $X$-ray transform of the characteristic function $\chi_K$ and evaluates the length of the chord
$K \cap (t\xi +\xi^{\perp}).$  This chord intersects $K$ of and only if $A_K(\xi,t) \neq 0.$

\begin{theorem}\cite{Ag2}\label{T:ellipse} Let $K$ be a domain in $\mathbb R^2$ with $C^{\infty}$ algebraic boundary. If the chord length function $A_K(\xi,t)$ has the form
$A_K(\xi,t)=\sqrt[m]{P_{\xi}(t)}$ (as long as $A_K(\xi,t) \neq 0$),  where $P_{\xi}$ is a polynomial, then $\partial K$ is an ellipse.
\end{theorem}

Here is the sketch of the proof. The first step is to understand the degree of the polynomial $P_{\xi}$ for a generic direction $\xi,$ by establishing lower and upper bounds for it.

The lower bound relies on
the fact that at the Morse points $a \in \partial K$ with the unit normal $\nu_a$ the chord length function $A_K(\nu_a, t)$ behaves, near the tangent lines to the boundary, as $const \  (h_K^+(\nu_a)-t)^{\frac{1}{2}},$ where $h_K^{+}(\xi)=
\sup_{x \in K} \langle x , \xi \rangle $ is the support function. There are two tangent lines orthogonal to
a given unit vector $\xi,$ namely, $  \langle x , \xi \rangle =h_K^{+}(\xi)$ and $ \langle x , \xi \rangle  =h_K^{-}(\xi),$ where $h_K^{-}(\xi)=\inf_{x \in K}  \langle x , \xi \rangle =-h_K^+(-\xi).$ Therefore $P_{\xi}(t)=A_K^m(\xi,t)$ vanishes at the points $t_{\pm}=h^{\pm}_K(\pm \xi),$ to the order $\frac{m}{2}$ at each. Thus, $\deg \ P_{\xi} \geq m$ when $\xi$ is a normal vector at Morse points. Since Morse points are dense on the boundary, the estimate holds everywhere, by continuity.

The upper bound for $\deg \ P_{\xi}$ can be obtained from the estimate of the growth of $A_K (\xi,t)$ as $ t \to \infty.$ However, moving $t$ in the real direction does not help since when $t$ is so large as the line $t\xi+\xi^{\perp}$ does not meet $K$ then $A_K(\xi,t)=0$ and the relation between $A_K(\xi,t)$ and $P_{\xi}(t)$ is lost.
 The idea is to move $t$ to the infinity in a complex direction. To this end, we need to analytically extend the chord length function $A_K(\xi,t)$ into the complex plane. By the uniqueness theorem, this extension must coincide with the algebraic function $\sqrt[m]{P_{\xi}(t)}.$ The analytic extension of $A_K(\xi,t)$ relies on our assumption that $\partial K$ is algebraic and hence is a trace on $\mathbb R^2$ of a complex algebraic curve in $\mathbb C^2.$ Then we construct a path connecting a real $t$ and $\infty$ and providing a single-valued analytic extension $A_K(\xi,t)$ along it, and prove that the order of growth of $A_K(\xi,t)$ along this path is at most one. This results in the estimate $\deg P_{\xi}(t) \leq m.$ Thus, for generic $\xi$ the degree of the polynomial $P_{\xi}$ is exactly $m$ and therefore $P_{\xi}(t)$ has the only zeros $t_{\pm}=h_K^{\pm} (\xi),$ each of multiplicity $\frac{m}{2}.$
Therefore, we have $P_{\xi}(t)=c(\xi) \big(h_K^+(\xi)-t \big)^{\frac{m}{2}} \big(t- h_K^{-}(\xi) \big)^{\frac{m}{2}}.$

 The second step relies on using the range description of the $X$-ray transform $A_K(\xi,t)$ (see, e.g., \cite[Chapter 1, Theorem 2.4]{He}).
 It yields that the moment of the degree $s$:
 $$M_s(\xi)=\int_{\mathbb R} A_K(\xi,t)t^s dt=c(\xi) \int_{h_K^{-}(\xi)} ^{h_K^+(\xi)}  \big(h_K^+(\xi)-t \big) ^{\frac{1}{2}}
 \big (t-h_K^{-}(\xi) \big)^{\frac{1}{2}}t^s dt $$
 is the restriction to the circle $|\xi|=1$ of a homogeneous polynomial of the degree  $s.$
 It is not difficult to obtain from here, using only the moments $M_s(\xi)$ of the degrees $s=0,1,2,$ that, after a suitable translation of $K$, the support function $h_K^{+}(\xi)$ is the restriction to $|\xi|=1$ of the square root of a quadratic polynomial. This yields that $\partial K$ is an ellipse.

\begin{remark} \rm The obtained result can be interpreted as follows. While Newton's Lemma about ovals \cite{ArVas} states that the cutoff function $V_K^{\pm}(\xi,t)$ of any domain $K \subset \mathbb R^2$ (with $C^{\infty}$ boundary) is a transcendental function, Theorem \ref{T:ellipse} specifies that among those transcendental functions, the Abel integrals $\int \sqrt[m]{P_{\xi}(t)} dt$ ($P_{\xi}$ is a polynomial) detect ellipses.
\end{remark}

\section{Radon transforms supported in hypersurfaces} \label{S:S7}

\subsection{The Interior Problem in Tomography}
In odd dimensions the inversion formula for the Radon transform is local in the sense that the value at $x$ of the function $f$ can be computed from the restriction of  the Radon transform $\R f(L) = \int_L f \, dx$ ($dx$ denotes surface measure on $L$) to the set of hyperplanes $L$ that intersect an arbitrarily small neighborhood of the point $x$.  As is well known, this is not the case in even dimensions. On the contrary, the so-called ``Interior problem'' for the even-dimensional  Radon transform is not solvable. In the $2$-dimensional case this statement is usually understood to mean the following. Let $D$ be an open disk in the plane and let $K$ be a closed subset of $D$. Then a function $f$ supported in $\overline D$, the closure of $D$,  cannot be determined anywhere in $K$ from the knowledge of its integrals over all lines that intersect $K$.
The proof of this statement runs as follows; see e.g., \cite{N}. Choose coordinates so that $D$ is a unit disk centered at the origin and let $D_0$ be a smaller concentric disk   such that $K \subset D_0$. Then look for radial functions $f(r)$, $r = |x|$,  that are supported in $r \le 1$  and satisfy $\R  f(\xi, p) = \R  f(p) = g(p) = 0$ for $|p| \le p_0 <1$. This   leads to  an Abel integral equation in which $f(r)$ can be solved in terms of $g(p)$, and the support of the solution $f$ is in general equal to $\overline D$.
By an affine transformation we see that the same is  true if the disk $D$ is replaced by the region bounded by an ellipse. However, if $D$ is an arbitrary bounded, convex domain, this argument does not work, and as far as we know it is not known if the corresponding statement is true.
% In fact, in \cite{B2}  the following conjecture was formulated.

\begin{conjecture}   \label{intprobl}
Let $D$  be an open, bounded, convex domain in the plane and $K$ be a closed subset of $D$.
Then there exists a function $f\in C_c^{\infty}(\bar D)$, not identically zero in $K$, such that its Radon transform $\R  f(L)$ vanishes for all lines  $L$ that intersect $K$.
\end{conjecture}

And what is the connection between the conjecture and Radon transforms supported in hypersurfaces?
Let us explain.
Denote by $f_0$   the function $\pi^{-1} (1 - |x|^2)_+^{-1/2}$, supported in the unit disk. An easy calculation shows that $\R  f_0(\xi, p) = 1$ for $|p| < 1$ and all $\xi$, and obviously $\R  f_0 = 0$ if $|p| > 1$. Hence the distribution $\partial_p \R  f_0$ is supported in the set of $(\xi, p)$ with $|p| = 1$, which corresponds to the set of lines that are tangent to the circle, and the same is true for $\partial_p^2 \R  f_0$.  Using the formula
\begin{equation}    \label{RDh}
(\R  \Delta h)(\xi, p) = \partial_p^2 \R  h(\xi, p)
\end{equation}
we can now conclude that the Radon transform of the distribution $f  = \Delta f_0$ is supported on the set of tangents to the unit circle ($\Delta$ is the Laplace operator). An affine transformation gives of  course similar examples where the disk is replaced by an elliptic region. A couple of years ago one of us (J.B.) got the idea to prove   Conjecture \ref{intprobl} by constructing analogous examples where $D$ is replaced by an  arbitrary convex region.

To be specific, assume for simplicity  that $D$ is symmetric, $D = -D$.
Choose an open, convex region $D_0$ with smooth boundary such that $D_0 = - D_0$,  $K \subset D_0$, and the closure $\overline D_0$ is contained in $D$.  Then try to find a distribution $f$, supported in $\overline D_0$, such that its Radon transform in the distributional sense, $\R  f = g$,  is supported  on the set of tangent lines  to the boundary $\partial D_0$.  A natural way to do this is to set
\begin{equation*}
g(\xi, p) = q(\xi) \delta  \big(p - h_{D_0}(\xi)\big) +  q(\xi) \delta  \big(p + h_{D_0}(\xi)\big) ,
\end{equation*}
where $h_{D_0}(\xi)$ is the support function for $D_0$,
\begin{equation*}
h_{D_0}(\xi) = \sup\{ \langle x , \xi \rangle ;\, x \in D_0\} , \quad \xi \in S^{1} ,
\end{equation*}
$\delta (\cdot)$ is the Dirac measure at the origin,
and try to choose the density function $q(\xi)$ so that $g$ is equal to $\R  f$ for some compactly supported distribution $f$. As is well known, the condition  for $g(\xi, p)$ to be in the range of $\R $ is that $g$ satisfies the moment conditions
\begin{align*} %    \label{rangecond}
\begin{split}
 \xi = (\xi_1, \xi_2)  & \mapsto \int_{\mathbb R} g(\xi, p) p^k dp  \quad
 \textrm{is a homogeneous polynomial}  \\
& \textrm{in  $(\xi_1, \xi_2)$ of degree $k$ for every $k\ge 0$} .
\end{split}
\end{align*}
\noindent
\parbox[t]{10cm}{
The problem then becomes  to determine the function $q(\xi)$ on the circle, such that the range conditions are fulfilled. And when  $f$ is found, we could regularize $f$ by convolving it with a smooth function $\phi(x)$ supported in a  small neighborhood of the origin, $f_1 = \phi * f$. If the support of $\phi$ is sufficiently small, then the support of $f_1$ will be contained in $D$ and $\R  f_1(L) = \int_L f_1 dx$ will vanish for all lines that intersect $K$ as desired. This idea reduces the problem to a one-variable problem just as the rotational symmetry did in the problem with two concentric disks discussed earlier.
}
\parbox[t]{4cm}{
 \begin{center}
\includegraphics[width=.2\textwidth] {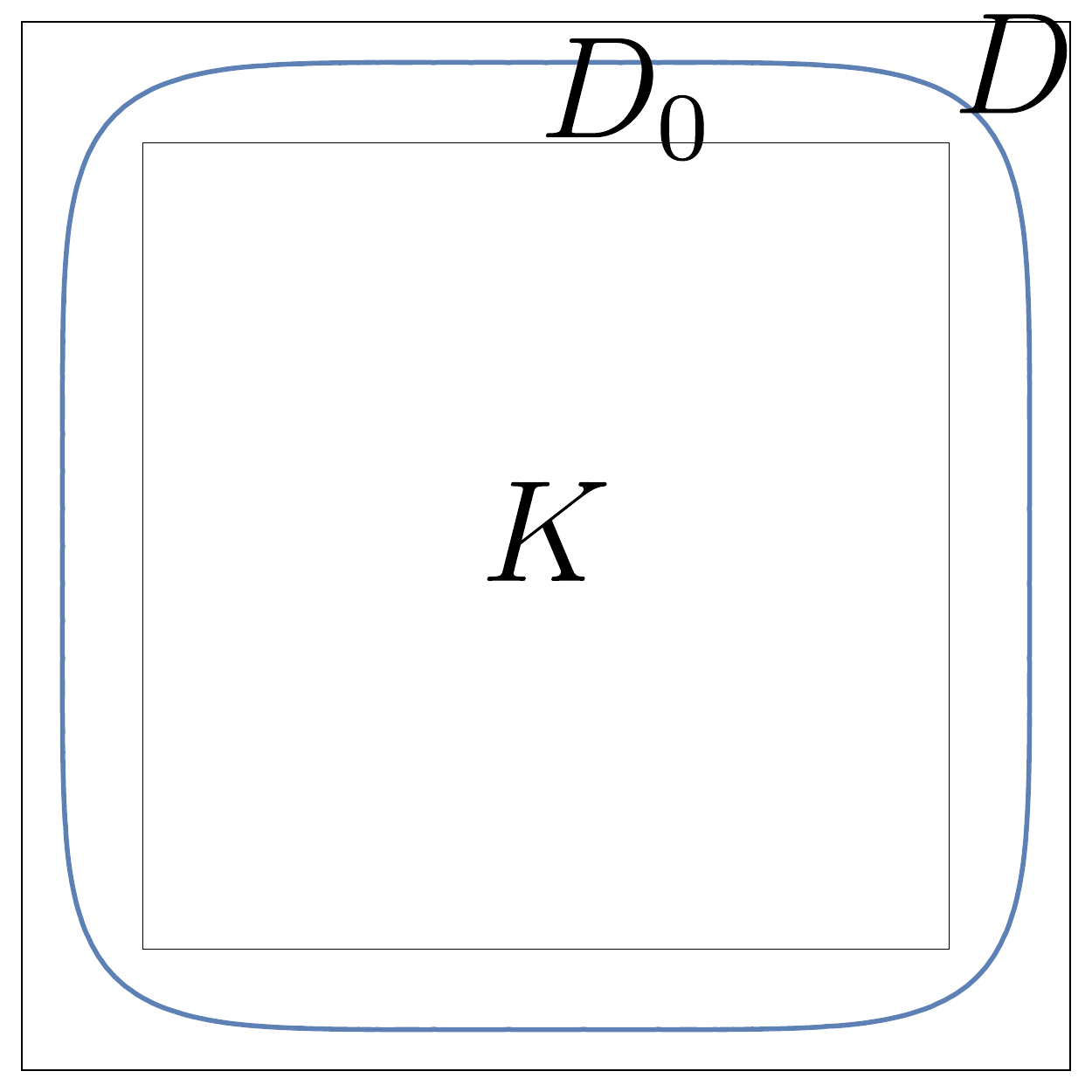}
\end{center}
}

\smallskip
However, surprisingly such distributions $f$ can exist only if $D_0$ is an ellipse.
In fact the following theorem holds, \cite{B1}, \cite{B2}.

\begin{theorem}   \label{jbthm}
Let $D$ be an open, convex,  bounded subset of $\mathbb R^n$    with boundary $\partial D$.
Assume that there exists a distribution $f$, supported in $\overline{D}$, such that the Radon transform of $f$ is  supported in the set of supporting planes to $D$. Then the boundary of $D$ is an ellipsoid.
\end{theorem}

    The next surprise was that this theorem gave a new proof of the result from \cite{KMY} which solved a special case of Arnold's problem, as described
    in Section \ref{S:S3}.
In fact, assume that the domain $D$ is polynomially integrable in the sense of Definition 1.2,
 and let
$\chi_D$ be the characteristic function of $D$. The assumption  that $D$ is polynomially integrable means that for every $\xi$ the function $p \mapsto \R  \chi_D(\xi, p)$ is a polynomial function in the interval  consisting of those $p$ for which the line $L(\xi, p) = \{x \in \mathbb R^n;\,  \langle x , \xi \rangle  = p \}$
intersects $D$.  Choose the integer $s$ so large that $2 s$ is greater than the degrees of all those polynomials.
Using repeatedly the formula \eqref{RDh}
with $h = \chi_D$ we see that the Radon transform of the distribution $\Delta^s \chi_D$ must vanish in the open set of lines that intersect $D$, as well as of course in the open set of lines that are disjoint from  $\overline D$. Hence $\R (\Delta^s f)$ is supported on the set of support planes to the boundary $\partial D$ of $D$.  But by Theorem \ref{jbthm} such distributions can exist only if $\partial D$ is an ellipse.

There is no smoothness assumption on the boundary $\partial D$ in Theorem~\ref{jbthm} ; therefore we have used the term supporting plane (to $D$) instead of tangent plane (to $\partial D$).
The assumption that $f$ is supported in $\overline D$ can be weakened to $f$ being compactly supported in
$\mathbb R^n$, because Helgason's support theorem \cite[Corollary 2.8]{He} shows that $f$ must vanish in the complement of the convex compact set $\overline D$, if $f$ is compactly supported and the Radon transform $\mathcal R f(L)$ vanishes for all hyperplanes $L$ that are disjoint from $\overline D$.

\subsection{On the proof of Theorem \ref{jbthm}}
Here we will prove Theorem \ref{jbthm} for the special case when the Radon transform of $f$ is a distribution of order   zero, and we will briefly indicate how the arguments given here   can be modified to cover the general case.

Let us begin by writing down an expression for an arbitrary distribution $g(\xi, p)$ of order zero   on the manifold $\mathcal P_n$ of hyperplanes in $\mathbb R^n$ that is supported on the set of supporting planes to $D$.

Since $L(\xi, p)$ and $L(-\xi, -p)$ are the same hyperplane, the distribution
 $g(\xi, p)$ must be even,  $g(\xi, p)  =  g(-\xi, -p)$.
A hyperplane $L(\xi, p)$ is a supporting plane for $D$ if and only if
\begin{equation*}
 p = h_D(\xi) \quad \textrm{or}  \quad
 p =  \inf\{ \langle x , \xi \rangle ;\, x \in D\} = - \sup\{-  \langle x , \xi \rangle ;\, x \in D\}
    = - h_D(-\xi).
\end{equation*}
An arbitrary distribution $g(\xi, p)$ on $S^{n-1}  \times \mathbb R$ of order zero that is supported on $\partial D$ can  therefore be written
\begin{equation*}
g(\xi, p)
  = q_+(\xi)\delta (p - h_D(\xi)) + q_-(\xi)\delta (p + h_D(-\xi))
\end{equation*}
for some functions or distributions   $q_+(\xi)$ and $q_-(\xi)$.
Since $\delta (t) = \delta (-t)$   we then have
\begin{align*}
g(-\xi, -p) & =  q_+(-\xi)\delta (-p - h_D(-\xi)) + q_-(-\xi)\delta (-p + h_D(\xi))  \\
& =  q_+(-\xi)\delta (p + h_D(-\xi)) + q_-(-\xi)\delta (p - h_D(\xi)) .
\end{align*}
The condition for $g(\xi, p)$ to be even  therefore becomes
\begin{equation*}
q_-(-\xi ) = q_+(\xi) \quad \textrm{for all $\xi$} .
\end{equation*}
Hence it is  sufficient to introduce one density function, say $q_+(\xi) = q(\xi)$, because then $q_-(\xi) = q(-\xi)$ (actually $q(\xi)$ must be a continuous function, \cite[Lemma 2]{B2}).
We conclude that an arbitrary distribution of order zero on the manifold $\mathcal P_n$  that is supported on the set of supporting planes to $D$ can be represented
\begin{equation*}    % \label{g}
g(\xi, p)  = q(\xi)\delta (p - h_D(\xi)) + q(-\xi)\delta (p + h_D(-\xi)) .
\end{equation*}
Observing that $\int_{\mathbb R} \delta (p \pm h_D(\xi)) p^k dp = (\mp h_D(\xi))^k$ we conclude that $q(\xi)$ and
$h_D(\xi)$  must satisfy the infinitely many equations
\begin{align}       \label{system1}
\begin{split}
& q(\xi) + q(-\xi) = p_0(\xi)  \\
& q(\xi) h_D(\xi) - q(-\xi) h_D(-\xi)  = p_1(\xi)   \\
& q(\xi) h_D(\xi)^2 + q(-\xi) h_D(-\xi)^2    = p_2(\xi)  \\
& q(\xi) h_D(\xi)^3 - q(-\xi) h_D(-\xi)^3    = p_3(\xi)  \\
& \phantom{xx} \ldots
\end{split}
\end{align}
for some homogeneous polynomials  $p_k(\xi)$  in $\xi = (\xi_1, \ldots , \xi_n)$ of degree $k$ for each $k$.
We have to prove that those identities imply that the boundary of $D$ is an ellipsoid.

If $D$ is a ball centered at the origin, then it is clear that $h_D(\xi) = c |\xi|$ for all $\xi \in S^{n-1}$ and some   constant $c$, hence $h_D(\xi)^2 = c^2 |\xi|^2$, which is a homogeneous quadratic polynomial.
It follows from the definition of $h_D(\xi)$ that a linear transformation $A$ that transforms $D$ to $\widetilde D = A D = \{A x;\, x \in D \}$ transforms
$h_D(\xi)$ to $h_{\widetilde D}(\xi) = h_D(A^*\xi)$.
Hence $h_D(\xi)^2$ must be a homogeneous quadratic polynomial whenever $D = - D$ and the boundary of $D$ is an ellipsoid.   This argument can obviously be reversed. Hence we conclude that for bounded  convex domains with $D  = - D$ holds
\begin{align}    \label{prop0}
\begin{split}
& \textrm{\raisebox{-1mm}{the boundary of $D$ is an ellipsoid, if and only if}} \\
& \textrm{\raisebox{1mm}{$h_D(\xi)^2$ is a homogeneous quadratic polynomial.}}
\end{split}
\end{align}
These observations make it possible to give a very short proof of the special case of Theorem \ref{jbthm} when  $D$ is symmetric and the distribution $g(\xi, p)$ has order zero.
Indeed, since $h_D(\xi)$ must be even, $h_D(\xi) = h_D(-\xi)$, we get  from the first and third equations in \eqref{system1}
\begin{equation*}
p_2(\xi)  =   \big(q(\xi)  + q(-\xi) \big)h_D(\xi)^2 = p_0(\xi) h_D(\xi)^2 .
\end{equation*}
But $p_0(\xi)  = q(\xi) + q(-\xi)$ must be equal to some constant $c$.   If $c \ne 0$, then  the fact that
$p_2(\xi)$ is a quadratic polynomial proves the assertion. If $c = 0$, then $q(\xi)$ is odd, so
\begin{equation*}
  q(\xi) h_D(\xi) - q(-\xi) h_D(\xi)  = 2 q(\xi) h_D(\xi) = p_1(\xi)
\end{equation*}
must be linear in $\xi$,  and
\begin{equation*}
p_3(\xi)  =   q(\xi) h_D(\xi)^3 - q(-\xi) h_D(-\xi)^3   = 2  q(\xi) h_D(\xi)^3
\end{equation*}
must be a homogeneous polynomial of degree $3$. Combining the last two equations we can write
\begin{equation*}
p_3(\xi) = h_D(\xi)^2 p_1(\xi) .
\end{equation*}
Since $h_D(\xi)$ is bounded, it follows that $p_3(\xi)$ must be divisible (in the polynomial ring)
 by $p_1(\xi) $, hence
$h_D(\xi)^2$ must be a quadratic polynomial, which completes the proof of Theorem \ref{jbthm} in   this case.

The case when the domain $D$ is not necessarily symmetric is somewhat more complicated.
We then have to consider the condition
\begin{equation*}   % \label{rho-}
h_D(\xi) h_D(-\xi) \quad \textrm{is a  polynomial}
\end{equation*}
instead of the condition that $h_D(\xi)^2$ is a polynomial. Note that  support functions are (positively) homogeneous of degree $1$, so if  the function $h_D(\xi) h_D(-\xi)$ is a polynomial, it must be a homogeneous quadratic polynomial.
Let $D_a = D + a$ with $a \in \mathbb R^n$ be the translated domain, and note that
$h_{D_a}(\xi) = h_D(\xi) +  \langle a , \xi \rangle $.
If $D = - D$ and $h_D(\xi)^2$ is a polynomial, then $h_{D_a}(\xi)^2$ is in general not a polynomial, but
\begin{equation*}
h_{D_a}(\xi)  h_{D_a}(-\xi) = \big(h_D(\xi) +  \langle  \xi,a \rangle \big)  \big(h_D(\xi) -  \langle   \xi,a \rangle \big)
= h_D(\xi)^2 -  \langle  \xi,a \rangle ^2
\end{equation*}
is a homogeneous quadratic polynomial for every $a \in \mathbb R^n$. This observation has the following important converse.

\begin{proposition}    \label{rho2}
Assume that $D$ is a convex, bounded domain for which the product $h_{D_a}(\xi)  h_{D_a}(-\xi)$
is a homogeneous quadratic polynomial for every $a \in \mathbb R^n$. Then the boundary of $D$ is an ellipsoid.
\footnote{Section 6 in \cite{B2} could have been omitted if we had known this fact at the time.}
\end{proposition}

\begin{proof}
Take an arbitrary $a \in \mathbb R^n\setminus\{0\}$, for instance $a = (1, 0, \ldots ,0)$. Then $ \langle a , \xi \rangle   = \xi_1$.
The assumption implies that
\begin{align*}
&  h_{D_a}(\xi) h_{D_a}(-\xi) -  h_D(\xi) h_D(-\xi) +  \xi_1^2    \\
= \big(h_D(\xi) & +  \xi_1 \big)  \big(h_D(-\xi) -  \xi_1 \big)
       - h_D(\xi) h_D(-\xi) +  \xi_1^2   \\
& =  \xi_1 \big(h_D(-\xi) -h_D(\xi) \big)
\end{align*}
is equal to a homogeneous quadratic polynomial $p_2(\xi)$.
This implies that the polynomial $p_2(\xi)$ is divisible by the linear factor $ \xi_1$ and hence the quotient must be another linear factor, so
\begin{equation*}
h_D(\xi) - h_D(-\xi) = - 2  \langle b , \xi \rangle
\end{equation*}
for some $b \in \mathbb R^n$.
 But this implies
\begin{align*}
& h_{D_b}(\xi) - h_{D_b}(-\xi)
= h_{D}(\xi) +  \langle b , \xi \rangle  - \big(h_{D}(-\xi)  -   \langle b , \xi \rangle  \big)  \\
& = h_{D}(\xi) - h_{D}(-\xi) + 2  \langle b , \xi \rangle   = - 2  \langle b , \xi \rangle   + 2  \langle b , \xi \rangle = 0  ,
\end{align*}
which shows that $h_{D_b}(\xi)$ is even.
Since $D$, and hence $D_b$, is convex, $D_b$ is uniquely determined by its support function, and it follows that $D_b$ is symmetric with respect to the origin.  By the assumption
$h_{D_b}(\xi) h_{D_b}(-\xi)$ is a homogeneous quadratic polynomial, so
$h_{D_b}(\xi)^2$ has the same property, and by \eqref{prop0} this implies that
the boundary of $D_b$ is an ellipsoid, and hence so is the boundary of $D$.
\end{proof}

We can now complete the proof of Theorem \ref{jbthm} for the case when $g(\xi, p)$ is a distribution of order zero. However, instead of writing the proof as short as possible we have chosen to present the calculations in a way that rather easily can be generalized to the case when $g$ is a distribution of higher order. We have to
show that the equations \eqref{system1} imply that the boundary of $D$ is an ellipsoid.

To shorten formulas we will   write
\begin{equation*}
\quad h_D(\xi) = h, \quad h_D(-\xi) = \check h,  \quad \textrm{and} \quad
q(\xi) = q, \quad q(-\xi) = \check q.
\end{equation*}
The infinite system \eqref{system1} can then be written in matrix form:
\begin{equation}    \label{system2}
\begin{pmatrix}
1  & 1  \\
h  & -  \check h \\
h^2  &  \check h^2 \\
h^3  & -  \check h^3 \\
\ldots & \ldots
\end{pmatrix}
\begin{pmatrix}
q \\
\check q
\end{pmatrix}
=
\begin{pmatrix}
p_0 \\
p_1 \\
p_2 \\
p_3  \\
\ldots
\end{pmatrix}  .
\end{equation}
Denote the sequence of $2 \times 2$ submatrices  of the big matrix to the left by $M_0, M_1, $ etc. Introduce  the column vectors
\begin{equation}   \label{Q}
Q = \begin{pmatrix}
q \\ \check q
 \end{pmatrix} ,
 \qquad
% \quad
 P_0 =
\begin{pmatrix}
p_0  \\ p_1
\end{pmatrix} , \quad
 P_1 =
\begin{pmatrix}
p_1  \\ p_2
\end{pmatrix} ,  \quad
P_2 =
\begin{pmatrix}
p_2  \\ p_3
\end{pmatrix} ,  \quad
\textrm{etc.}
\end{equation}
Then we have the equations $M_0 Q = P_0$, $M_1 Q = P_1$, etc., and more generally
\begin{equation}   \label{Mkq}
M_k Q = P_k   \quad  \textrm{for all $k \ge 0$} .
\end{equation}
The matrices $M_k$ form a geometric series in the sense that
\begin{equation}      \label{Mk+1}
M_{k+1} = S M_k = M_k T \quad  \textrm{for all $k \ge 0$}
\end{equation}
with
\begin{equation*}
S =
\begin{pmatrix}
0   & 1  \\
h \check h & h -  \check h
\end{pmatrix}
\quad \textrm{and} \quad
T =
\begin{pmatrix}
h   & 0  \\
0 &   -  \check h
\end{pmatrix} .
\end{equation*}
This makes it easy to eliminate $Q$ from the system \eqref{system2}. In fact
\begin{equation}    \label{SP0}
S P_0 = S M_0 Q = M_1 Q = P_1
\end{equation}
and similarly
\begin{equation}   \label{SPk}
SP_k = P_{k+1} \quad \textrm{for all $k \ge 0$} .
\end{equation}
Viewing a row of columns as a matrix we can then form matrix identities  by combining pairs of equations \eqref{SPk} as follows
\begin{equation}    \label{SPP}
S(P_0, P_1) = (P_1, P_2)   ,
\end{equation} and more generally
\begin{equation}  \label{SkPP}
S^k(P_0, P_1) = (P_k, P_{k+1})\quad \textrm{for all $k \ge 0$} .
\end{equation}
Using the product law for determinants in \eqref{SPP} we see that
$h \check h = - \det S$ must be a rational function, provided
\begin{equation}    \label{det}
\det (P_0, P_1) = p_0 p_2 - p_1^2 \ne 0
\end{equation}
as a polynomial.
Assuming \eqref{det} for a moment we prove that $h \check h = - \det S$ must be a polynomial by applying
\eqref{SkPP} for sufficiently large $k$ instead of \eqref{SPP}, which leads to a contradiction unless $h \check h$ is a polynomial.
To prove \eqref{det}  we first note that the following identity of matrices is valid
\begin{equation}    \label{P0P1}
(P_0, P_1) = (M_0 Q, M_1 Q) = (M_0 Q, M_0 T Q) = M_0  ( Q, T Q) .
\end{equation}
The matrix $M_0$ is non-singular for all $\xi$, since
\begin{equation*}
\det M_0 = - (h + \check h) ,
\end{equation*}
$h(\xi)$ is strictly positive for all $\xi$ if the origin is contained in $D$, and the expression
$h(\xi) + h(-\xi)$ is invariant under translations of the coordinate system. Moreover
\begin{equation}    \label{detQTQ}
\det  ( Q, T Q) = \det
\begin{pmatrix}
q   &  h \, q  \\
\check q &   -  \check h \,\check q
\end{pmatrix}
= - q \check q (h + \check h)  .
\end{equation}
It remains to show that $q \check q$ cannot be identically zero.

\begin{lemma}   \label{q2}
Assume that \eqref{system1} holds and that $q(\xi) \ne 0$ for some $\xi$.
Then  $q(\xi)q(-\xi) \ne 0$ for some $\xi$.
\end{lemma}

\begin{proof}
Solving $q$ and $\check q$ from the first two equations in the system \eqref{system2}  gives
\begin{equation}    \label{q}
q =  (p_0 \check h + p_1)/(h + \check h) , \qquad \check q =  (p_0 h - p_1)/(h + \check h)  .
\end{equation}
Since $q$ is assumed not to vanish identically, already the first equation of  \eqref{q} shows that  $p_0$ and $p_1$ cannot both vanish identically. Choosing coordinate system so that the origin is in $D$ we have
$h = h_D(\xi) > 0$ for all $\xi$ and hence  $h \check h > 0$ for all $\xi$.
A translation of the coordinate system does not change $q(\xi)$.
Recall that $p_0$ is constant and that $p_1$ is a homogeneous first degree polynomial with real coefficients.
Choose $\xi^0$ so that $p_1(\xi^0) = 0$. Then
\begin{equation*}
 q \check q =  p_0^2 h \check h/(h + \check h)^2  > 0 \quad \textrm{at $\xi^0$}
\end{equation*}
 if $p_0 \ne 0$. On the other hand, if $p_0 = 0$, then by \eqref{q} we have $q \check q =  - p_1^2/(h + \check h)^2$, which cannot be identically zero.
\end{proof}

By means of Lemma \ref{q2} together with \eqref{P0P1} and \eqref{detQTQ} we can now conclude that
the polynomial $\det (P_0, P_1)$ cannot be identically zero, and hence by \eqref{SPP} that
$\det S = - h \check h$ must be a rational function.
And, as already mentioned,  using \eqref{SkPP} for arbitrarily large $k$ we now obtain a contradiction unless
$h \check h$ is a polynomial.
The same must be true if the domain $D$ is replaced by the translate $D + a$ for arbitrary $a \in \mathbb R^n$, because
\begin{align*}
\int g(\xi, p)  (p +  \langle a , \xi \rangle )^k dp
& = \sum_{j=0}^k \binom kj \int g(\xi, p) p^j dp \, \langle a , \xi \rangle ^{k-j} \\
& = \sum_{j=0}^k \binom kj p_j(\xi)  \langle a , \xi \rangle ^{k-j} ,
\end{align*}
and each term in the last expression is clearly a homogeneous polynomial of degree $k$.
An application  of  Proposition \ref{rho2} now finishes the proof of Theorem \ref{jbthm} for the case when $g(\xi, p)$ is a distribution of order zero.

\smallskip

We now sketch the proof of Theorem \ref{jbthm}   for the case when $g(\xi, p)$ is a distribution of order at most $1$.
An arbitrary even distribution $g(\xi, p)$ on $S^{n-1} \times \mathbb R$ of order at most $1$ that is supported on
\begin{equation*}
\{(\xi, p);\, p = h_D(\xi)\} \cup \{(\xi, p);\, p = - h_D(-\xi)\}
\end{equation*}
can be written $g = g_0 + g_1$, where
\begin{align*}    % \label{g1}
\begin{split}
g_0(\xi, p)  & = q_0(\xi)\delta (p - h_D(\xi)) + q_0(-\xi)\delta (p + h_D(-\xi))  \\
g_1(\xi, p)  & = q_1(\xi)\delta '(p - h_D(\xi)) - q_1(-\xi)\delta '(p + h_D(-\xi))  .
\end{split}
\end{align*}
The minus sign between the terms in the expression for $g_1$ is needed to make $g_1(\xi, p)$ even, because $\delta '(\cdot)$ is odd. The matrix form of the system analogous to \eqref{system2} then becomes
\begin{equation}    \label{system3}
\begin{pmatrix}
1  & 0 & 1 & 0 \\
h  & - 1 &  - \check h & -1  \\
h^2  & - 2 h &  \check h^2  &  2 \check h   \\
h^3  & - 3 h^2 &  - \check h^3  & - 3 \check h^2 \\
h^4  & - 4  h^3  &  \check h^4  & 4 \check h^3 \\
\ldots & \ldots  & \ldots & \ldots
\end{pmatrix}
\begin{pmatrix}
q_0 \\
q_1 \\
\check q_0  \\
\check q_1
\end{pmatrix}
=
\begin{pmatrix}
p_0 \\
p_1 \\
p_2 \\
p_3  \\
p_4  \\
\ldots
\end{pmatrix}  .
\end{equation}
Denote the sequence of $4 \times 4$ submatrices of the big matrix in \eqref{system3} by $M_0, M_1,$ etc., and in analogy with \eqref{Q}  introduce the column vectors
\begin{equation*}
Q = (q_0, q_1, \check q_0, \check q_1)^t , \quad \textrm{and} \quad
P_k = (p_k, p_{k+1}, p_{k+2}, p_{k+3})^t , \quad k = 0, 1, \ldots.
\end{equation*}
Then \eqref{Mkq}, \eqref{Mk+1},  \eqref{SP0}, and \eqref{SPk}  are valid with
\begin{equation*}
S =
\begin{pmatrix}
0 & 1 & 0 & 0  \\
0 & 0 & 2 & 0  \\
0 & 0 & 0 & 3 \\
\sigma_4&  - \sigma_3 &\sigma_2 & -\sigma_1
\end{pmatrix} ,
\quad
T =
\begin{pmatrix}
h & 1 & 0 & 0  \\
0 & h & 0 & 0  \\
0 & 0 & -\check h & 1 \\
0 &  0  &0 & -\check h
\end{pmatrix}  ,
\end{equation*}
where
\begin{equation*}
\sigma_1 =  2(h - \check h) , \quad \sigma_2 =  - h^2  - 4 h \check h - \check h^2, \quad
\sigma_3 = -2 h \check h(h - \check h) , \quad \sigma_4 = - h^2 \check h^2 .
\end{equation*}
Note that $\sigma_j$ is up to sign equal to the   elementary symmetric polynomial in four variables of degree $j$, evaluated at
$(h, h, -\check h, -\check h)$.
In this case
\begin{equation*}
\det S = \det T = h^2 \check h^2 = h(\xi)^2 h(-\xi)^2 .
\end{equation*}
We can eliminate $Q$ in the same way as before.
Forming $4 \times 4$ matrix identities analogous to \eqref{SPP} and \eqref{SkPP} we prove in the same way as before that $h^2 \check h^2$ is a rational function and in fact a polynomial, provided the determinant
\begin{equation}    \label{det1}
\det(P_0, P_1, P_2, P_3) \quad \textrm{is different from the $0$-polynomial.}
\end{equation}
It remains only to prove \eqref{det1}.
A calculation \cite[(4.17)]{B2} shows that
\begin{equation*}
\det(P_0, P_1, P_2, P_3)  = \det M_0 \det(Q, TQ, T^2Q, T^3Q) = c \, q_1  \check q_1 (h + \check h)^4
\end{equation*}
with $c \ne 0$, so it is enough to prove the analogue of Lemma \ref{q2} showing that
$q_1  \check q_1 = q_1(\xi) q_1(-\xi)$ cannot be identically zero.
However,  we have not found an easy and elementary proof of this fact for the  case when $g$ is a distribution of higher order. One way is to use the -- admittedly rather awkward  -- Lemma 5.2 from \cite{B2}. A more elegant, but less elementary, way is to use the result from \cite{B3} showing that $q_1(\xi)$ must be real analytic \cite[Theorem 2]{B3}.  This implies that $q_1(\xi) q_1(-\xi)$ cannot be identically zero, since $q_1(\xi)$ is assumed not to be identically zero. This completes the sketch of the proof of Theorem \ref{jbthm} when $  f$ is a distribution of order at most $1$.

\subsection{Singularities of the boundary of the support and singularities of the distribution}
From a different point of view one could say that the results of  \cite{B1} and \cite{B2} inferred information about the regularity of the support of the distribution  $g = \R  f$  from regularity properties of the distribution itself. Indeed, the fact that a hypersurface $\Sigma$ in the manifold $\mathcal P_n$ of hyperplanes in $\mathbb R^n$ is the set of tangent planes to an ellipsoid is an expression of very strong regularity of $\Sigma$.
And the assumption that $\xi \mapsto \int g(\xi, p)p^k dp$ is a polynomial for every $k$ implies a microlocal regularity property of $g$. In fact, already the weaker property that $\xi \mapsto \int g(\xi, p)\phi(p) dp$ is real analytic in a neighborhood of $\xi^0$ for every real analytic $\phi(p)$ is equivalent to
every conormal of the line $\gamma_{\xi^0}: \ p \mapsto (\xi^0,p)$ being absent in the analytic wave front set of $g$, in H\"ormander's notation $W\!F_A(g) \cap N^*(\gamma_{\xi^0}) = \emptyset$. Here $W\!F_A(g)$ denotes the analytic wave front set of $g$, and $N^*(\gamma_{\xi^0})$ denotes the set of conormals in $T^*(\mathcal P_n)$ to the line  $\gamma_{\xi^0}$. Geometrically $\gamma_{\xi^0}$ is the set of all hyperplanes that are orthogonal to $\xi^0$, which is of course a hypersurface in $\mathcal P_n$. Our arguments in \cite{B1} and \cite{B2}, as briefly sketched above,  can be used to prove the following {\it local} statement.   If a distribution $g$ (which need not be a Radon transform) is assumed to be supported on a hypersurface $\Sigma$ in a real analytic manifold, and $\gamma$ is a smooth curve that intersects $\Sigma$ transversally, then $W\!F_A(g) \cap N^*(\gamma) = \emptyset$  implies that the surface $\Sigma$ is real analytic in a neighborhood of the intersection point, and more.
A theorem of this kind was presented in \cite{B3}.   This result is closely related to a key step in H\"ormander's famous proof of Holmgren's uniqueness theorem.

%\bibliographystyle{amsplai\acute{}n}
%    Insert the bibliography data here.

%\noindent
%Bar-Ilan University and Holon Institute of Technology;  Israel.

%\noindent
%{\it E-mail address:} agranovs@math.biu.ac.il

\end{document}